\documentclass{article}

\usepackage[a4paper,margin=.8in]{geometry}

\usepackage{url}
\usepackage{amsmath,amsthm,amsfonts,amssymb,amscd,mathtools,amsopn}
\usepackage{graphicx}
\usepackage{epstopdf}      
\usepackage{svg}           
\usepackage{float}         
\usepackage{comment}       
\usepackage{enumitem}      
\usepackage{algorithm}     
\usepackage{algpseudocode} 

\usepackage{notes2bib}     
\usepackage{xcolor}        
\usepackage{lipsum}        

\ifpdf
  \DeclareGraphicsExtensions{.eps,.pdf,.png,.jpg}
\else
  \DeclareGraphicsExtensions{.eps}
\fi

\setlist[enumerate]{leftmargin=.5in}
\setlist[itemize]{leftmargin=.5in}

\newtheorem{cor}{Corollary}
\newtheorem{thm}{Theorem}
\newtheorem{lem}{Lemma}

\newtheorem{define}{Definition}
\newtheorem{fixedtheorem}{Theorem}

\newcommand{\spans}{\operatorname{span}}
\newcommand{\st}{\,|\,}

\newcommand{\gattract}{\mathbb{A}}

\begin{document}

\title{Interplay between Contractivity and Monotonicity for Reaction Networks
}

\author{Alon~Duvall\thanks{Department of Mathematics, Northeastern University, Boston, MA (\texttt{duvall.a@northeastern.edu}).}
\and M.~Ali~Al-Radhawi\thanks{Department of Electrical and Computer Engineering and Department of Bioengineering, Northeastern University, Boston, MA (\texttt{malirdwi@northeastern.edu}).}
\and Dhruv~D.~Jatkar\thanks{Department of Bioengineering, Northeastern University, Boston, MA (\texttt{jatkar.d@northeastern.edu}).}
\and Eduardo~D.~Sontag\thanks{Department of Electrical and Computer Engineering and Department of Bioengineering, Northeastern University, Boston, MA (\texttt{e.sontag@northeastern.edu}).}}

\maketitle

\begin{abstract}
 We establish a new relationship between monotonicity and contractivity and use this connection to describe a new general class of weakly contractive reaction networks. The new class is characterized by the stoichiometry matrix of the reaction network admitting a precise matrix factorization that can be verified computationally. Reaction networks in this class are weakly contractive, implying global convergence to equilibria under appropriate technical conditions. Furthermore, we describe the novel subclass of cross-polytope networks. We also show that our results provide a unified proof of global convergence for several classes of networks previously studied in the literature. The practical relevance of the results is demonstrated by examples from systems biology and signaling pathways. 
\end{abstract}

\section{Introduction}

Formal models known as reaction networks provide useful descriptions of various processes, such as in biology and chemistry. Examples include modeling DNA replication kinetics \cite{https://doi.org/10.1002/bit.20617}, futile cycles in molecular systems biology \cite{Angeli2010-hj}, the Extracellular Regulated Kinase signaling pathway \cite{10.1371/journal.pcbi.1007681}, and post-translational modification cycles \cite{cite-key}. For general introductions to reaction network theory, see \cite{erdi1989mathematical} or \cite{Feinberg}.

A fundamental problem in these models is predicting qualitative behavior. Significant effort has gone into developing mathematical frameworks and theorems describing their qualitative dynamics. Key results include deficiency theory for stability and uniqueness of equilibria \cite{Feinberg}, conditions for multistationarity \cite{doi:10.1137/15M1034441}, and bifurcation criteria \cite{Banaji_2023}. Of particular interest is whether a reaction network's dynamics converge to a unique equilibrium. We focus on two properties closely tied to convergence: contractivity, which ensures that trajectories move closer together over time, and monotonicity, which imposes an ordering on system dynamics, constraining how solutions evolve relative to each other.

Both contractivity and monotonicity have extensive literature \cite{smith2008monotone} \cite{7039986} \cite{LOHMILLER1998683} \cite{9799744}, \cite{SIMPSONPORCO201474} with applications to safety constraints \cite{jafarpour2023monotonicity}, network systems \cite{9403888}, excitatory Hopfield neural networks \cite{9961864}, and biological networks \cite{alradhawi2023structural}. These methods have also been used separately to analyze reaction networks \cite{article}, \cite{Angeli2010-hj}, \cite{BANAJI20131359}, \cite{doi:10.1137/120898486}, \cite{doi:10.1080/14689360802243813}. Contractivity has been applied to reaction networks, such as in \cite{article}, where a certain class of networks is shown to have a negative logarithmic norm, implying contractive properties. Meanwhile, \cite{Angeli2010-hj}, \cite{BANAJI20131359}, and \cite{doi:10.1137/120898486} use monotonicity to establish convergence. Contractivity has also been linked to Lyapunov functions: \cite{7097666} presents algorithms for constructing piecewise linear Lyapunov functions for reaction networks, and later work connects these functions to contractivity \cite{alradhawi2023structural}.

In this work, we establish a connection between monotonicity and contractivity, allowing us to prove weak contractivity (defined below) for a class of reaction networks via monotonicity. Weak contractivity limits the possible long-term behaviors of a system. Notably, in several realistic biochemical systems, weak contractivity guarantees that each stoichiometric compatibility class admits a unique attracting equilibrium. Our main results establish weak contractivity for new classes of reaction networks. In particular, we establish weak contractivity for the globally convergent networks considered in \cite{doi:10.1137/120898486}, \cite{BANAJI20131359}, \cite{Angeli2010-hj}, \cite{leenheer_monotone_reaction_networks}, and \cite{EITHUN20171}. Additionally, in \cite{doi:10.1137/120898486}, the authors suggest that unifying the notions of monotonicity in species coordinates and reaction coordinates is an important direction for future research. We take a step in this direction by showing that any reaction-coordinate cooperative network can be made species-coordinate monotone through the addition of a single dummy species.

\subsection{New class of networks} We now state one of our core theorems, which generalizes several results from the literature. Before doing so, we introduce a few definitions.

Let \( S \) be the set of real matrices (of any dimension) whose entries lie in \( \{ -1, 0, 1 \} \), and in which each column has at most two nonzero entries. Define \( S^t \) to be the set of matrices whose transposes belong to \( S \); that is, \( A \in S^t \) if and only if \( A^t \in S \). Let \( \mathcal{N} = S \cup S^t \).

Let \( \mathcal{P} \) denote the set of matrices with at most one nonzero entry in each row and no column of all zeros. Let \( \mathcal{D} \) denote the set of square diagonal matrices with positive diagonal entries.

We say a reaction network is \textit{non-catalytic} if no species appears on both sides of any reaction. We say that a reaction network is weakly contractive if it is weakly contractive on each stoichiometric compatibility class. (Undefined terms are introduced in Section~\ref{sec:background}.)

\begin{fixedtheorem}
\label{thm:collection of theorems}
Let a non-catalytic reaction network have a stoichiometric matrix of the form \( PND \), where \( P \in \mathcal{P} \), \( N \in \mathcal{N} \), and \( D \in \mathcal{D} \), and where the matrix product \( PND \) is well-defined. If, in addition, the \( \mathcal{RI} \)-graph of the network is strongly connected, then the corresponding system is weakly contractive.
\end{fixedtheorem}

We provide here an informal sketch of the proof strategy. It begins by using Theorem~\ref{thm:output_M_to_strong_monotone}, which provides an efficient way to check whether a given reaction network is monotone with respect to a polyhedral cone. As shown in Section~\ref{sec:connect_montone_contractive}, this also offers a method for verifying non-expansivity. The remainder of the proof proceeds by partitioning the class of reaction networks into several cases and, for each case, explicitly constructing a norm with respect to which the system is weakly contractive.

We focus on the matrix \(N\) in the factorization \(PND\), since Theorem~\ref{thm:P*gamma network strongly monotone} guarantees that if \(N\) is weakly contractive, then \(PN\) is as well. The matrix \(D\) does not affect the class of dynamical systems under consideration and can therefore be ignored.

The three structural cases for \(N\) are:
\begin{enumerate}
    \item Each column of \(N\) has at most two nonzero entries (Corollary~\ref{cor:type C is weakly contractive}),
    \item Each row has at most two nonzero entries and the columns are linearly independent (Theorem~\ref{thm:cubicals are weakly contractive}), and
    \item Each row has at most two nonzero entries and the columns are linearly dependent (Theorem~\ref{thm:dependent reaction monotone is contractive}).
\end{enumerate}

The latter two results are unified in Theorem~\ref{2 in each row aligned weakly contractive}.

This theorem covers all the explicit weakly contractive network classes considered in this paper. A key advantage is that for many networks, its applicability can be easily determined via inspection. This theorem applies to several classes of reaction networks studied in the literature, which we refer to as the Type~I, Type~A, and Type~S classes. Figure~\ref{fig:myimage} provides a graphical illustration of how these previously known results arise as special cases of our main theorem. The figure also highlights how certain new classes of reaction networks fall under the scope of our theorem—specifically, cases that, to the best of the authors' knowledge, have not previously had their asymptotic behavior analyzed (such as when \( N \) has at most two nonzero entries per column and linearly independent rows).

\subsection{Two examples} We now demonstrate that several realistic biological networks fall under the scope of our main theorem, more examples will be shown in \S \ref{sec:biochemical examples}. The following examples are adapted from \cite{10.1371/journal.pcbi.1007681}.

Consider a network modeling two species, $E_1$ and $E_2$, competing to bind with another species $X$:
\[
X + E_1 \rightleftharpoons XE_1, \quad X + E_2 \rightleftharpoons XE_2,
\]
with the associated stoichiometric matrix:
\[
\begin{bmatrix}
    -1 & -1 \\
    -1 & 0 \\
    0 & -1 \\
    1 & 0 \\
    0 & 1 \\
\end{bmatrix}
= 
\begin{bmatrix}
    1 & 0 & 0 \\
    0 & -1 & 0 \\
    0 & 0 & -1 \\
    0 & 1 & 0 \\
    0 & 0 & 1 \\
\end{bmatrix}
\begin{bmatrix}
    -1 & -1 \\
    1 & 0 \\
    0 & 1 \\
\end{bmatrix}
\begin{bmatrix}
    1 & 0 \\
    0 & 1 \\
\end{bmatrix}.
\]
This admits a factorization of the form $PND$ with $N \in S \cup S^t$, so our results apply and the network is weakly contractive. As we will see later, this leads to strong dynamical conclusions.

As a second example, consider \textit{three-body binding}, where two molecules, $X$ and $Y$, can bind only in the presence of a third molecule $E$, which facilitates the interaction. This scenario commonly arises in biological systems. A reaction network model capturing this behavior is:
\[
X + E \rightleftharpoons XE, \quad Y + E \rightleftharpoons YE,
\]
\[
YE + X \rightleftharpoons XEY, \quad Y + XE \rightleftharpoons XEY.
\]

The corresponding stoichiometric matrix and its factorization are:
\[
\begin{bmatrix}
    -1 & -1 &  0 &  0 \\
    -1 &  0 & -1 &  0 \\
     0 & -1 &  0 & -1 \\
     1 &  0 &  0 & -1 \\
     0 &  1 & -1 &  0 \\
     0 &  0 &  1 &  1 \\
\end{bmatrix}
=
\begin{bmatrix}
    1 & 0 & 0 & 0 & 0 & 0 \\
    0 & 1 & 0 & 0 & 0 & 0 \\
    0 & 0 & 1 & 0 & 0 & 0 \\
    0 & 0 & 0 & 1 & 0 & 0 \\
    0 & 0 & 0 & 0 & 1 & 0 \\
    0 & 0 & 0 & 0 & 0 & 1 \\
\end{bmatrix}
\begin{bmatrix}
    -1 & -1 &  0 &  0 \\
    -1 &  0 & -1 &  0 \\
     0 & -1 &  0 & -1 \\
     1 &  0 &  0 & -1 \\
     0 &  1 & -1 &  0 \\
     0 &  0 &  1 &  1 \\
\end{bmatrix}
\begin{bmatrix}
    1 & 0 & 0 & 0 \\
    0 & 1 & 0 & 0 \\
    0 & 0 & 1 & 0 \\
    0 & 0 & 0 & 1 \\
\end{bmatrix}.
\]

We write out the factorization explicitly in the form $PND$, though we could just as well have observed that the stoichiometric matrix contains only entries in $\{-1, 0, 1\}$ and has at most two nonzero entries in each row. Thus, our main theorem again applies, and the network is weakly contractive.

\begin{figure}[ht]
    \centering
    \includegraphics[width=0.8\textwidth]{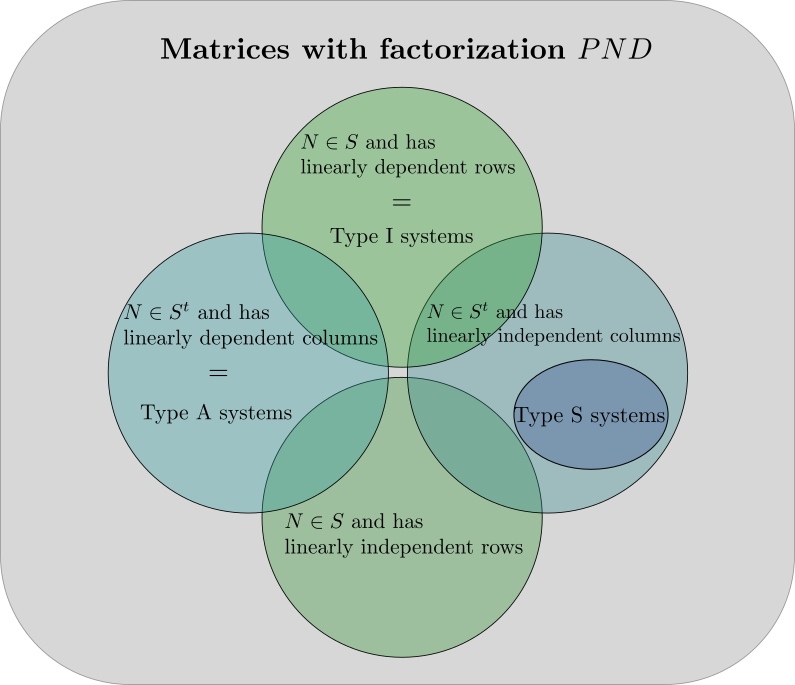}
    \caption{Venn diagram describing various relationships between classes of weakly contractive systems. Type I networks are from \cite{doi:10.1137/120898486} and are discussed in Section \ref{sec:type_I_networks}, Type S networks are from \cite{BANAJI20131359} and are discussed in Section \ref{sec:cubical_cones}, and Type A networks are from \cite{Angeli2010-hj} and are discussed in Section \ref{sec:react_coordinates}. 
    }
    \label{fig:myimage}
\end{figure}

\subsection{Paper structure}
Our work is organized as follows. In Section \ref{sec:background} we establish our background and notation, including a connection between monotonicity and contractivity, allowing us to go between the two properties. In Section \ref{sec:Conditions for monotone} we establish some preliminary lemmas to allow us to establish the monotonicity of a system. In Section \ref{sec:algo for cones} we introduce viable sets and prove additional sufficient conditions for strong monotonicity. Sections \ref{sec:cross-polytopes cones} and \ref{sec:additional weakly contractive systems} use our methods to explicitly show certain classes of reaction networks are weakly contractive with respect to certain norms. Finally, Section \ref{sec:proof of theorem 1} brings multiple parts together to establish Theorem \ref{thm:collection of theorems}. Section \ref{sec:biochemical examples} provides biochemical examples where our theorems apply.

\section{Background and Notation}
\label{sec:background}

We study reaction networks (defined below) and their associated differential equations. Fix a positive integer $n$ throughout the paper.

For a vector $x \in \mathbb{R}^n$ we indicate the $i$'th coordinate as $(x)_i$, and if $I$ contains multiple indices then $(x)_I$ is the ordered set of values with these indices. As an example, if we take $I = \{1,2,4\}$ and $x = [2,3,4,5,6]$, then we have $(x)_I = [2,3,5]$. The support of a vector $x$ is the set of indices $I$ such that if $i \in I$ then $(x)_i \neq 0$. We indicate the usual Euclidean inner product by $\langle x,y \rangle$. By ``cone," we mean a convex cone (closed under nonnegative scaling and addition). Given a cone $K \subseteq \mathbb{R}^n$ its dual cone is denoted by $K^* \subseteq \mathbb{R}^n$ and defined as $K^* = \{x \in \mathbb{R}^n| \langle x, y\rangle \geq 0 \ \forall \ y \in K \}$. The Jacobian of $f$ is denoted by $\mathcal{J}_f$. The boundary and interior of a set $B$ are denoted by $\partial B$ and $\mbox{Int}(B)$, respectively. Let $\mathbb{R}_{\geq 0}$ be the set of nonnegative real numbers. A dynamical system defined by $\dot{x} = f(x)$ is sometimes referred to simply as a system. The state space of a system is simply a forward invariant set for the system.

 For standard terminology related to polytopes and cones—such as extreme rays and faces of a cone—we refer the reader to Section~\ref{section_appendix_polytope_terms}. At times, we abuse notation by referring to an extreme ray using a single nonzero vector lying on the ray; the intended meaning should be clear from context. For a subspace $L$ we indicate the orthogonal subspace by $L^{\perp}$. In general, for a vector $v$ we indicate by $v^* = v^{\perp} \cap K^*$ the vector space orthogonal to $v$ intersected with the cone $K^*$ ($K^*$ is the dual of a cone $K$, where the definition of $K$ should be clear from context). For a face of a cone $F$ we define $F^* = (\spans \ F)^{\perp} \cap K^*$. In particular, if we have an extreme ray $v = k_1 \in K$ then $k_1^*$ is a face of $K^*$ called the dual face. The affine hull of a set $M \subseteq \mathbb{R}^n$ is $\mbox{AffHul}(M)= \left\{ \sum_{i=1}^k \alpha_i x_i | x_i \in M, k > 0, \alpha_i \in \mathbb{R}, \sum_{i=1}^k \alpha_i = 1 \right\}$. By the dimension of a set, we mean the dimension of the affine hull of the set. In general, when we consider the interior or boundary of a set $M$, we consider these operations with respect to the subspace topology of $\mbox{AffHul}(M)$. We sometimes use the phrase `relative interior' or `relative boundary' when using the interior or boundary operations in this sense. A cone \( K \) is said to be \textit{pointed} if \( K \cap (-K) = \{0\} \), and \textit{proper} if it contains a nonempty open set. A directed graph is \textit{strongly connected} if, for every pair of nodes, there exists a directed path from one to the other. The term \textit{poset} refers to a partially ordered set, that is, a set equipped with a binary relation \( \leq \) that is reflexive, antisymmetric, and transitive.

\begin{define}
    A \textit{reaction network} consists of \( m \) reactions, each labeled as either reversible or irreversible. Each reaction is represented by a column vector in \( \mathbb{R}^n \), called a \textit{reaction vector}, which describes the net change in species concentrations caused by that reaction.

\end{define}

These vectors form the columns of the stoichiometric matrix \( \Gamma \). The reversible/irreversible labels are not encoded in \( \Gamma \), and their implications are discussed later in the paper.

The reaction vectors can be collected into a matrix \( \Gamma \), called the \textit{stoichiometric matrix}, where each reaction vector forms a column of \( \Gamma \). Since the ordering of reactions does not affect our results, it is chosen arbitrarily.

\begin{define}
    The \textit{stoichiometric compatibility class} of an initial state \( S_0 \in \mathbb{R}^n_{\geq 0} \) is the set of all states reachable via the reaction network, given by:
    \[
    \{S_0 + \Gamma x \mid x \in \mathbb{R}^m\} \cap \mathbb{R}^n_{\geq 0}.
    \]
\end{define}

We will sometimes refer to the reaction network as the matrix $\Gamma$.
\begin{define}
    The \textit{species} of a reaction network correspond to the coordinates of the space $\mathbb{R}^n$ where each reaction vector resides.
    
\end{define}

\begin{define}
    Two reaction vectors $v_1$ and $v_2$ are said to \textit{share species $i$} if both $(v_1)_i$ and $(v_2)_i$ are nonzero. More generally, they \textit{share species $I$} if this holds for every $i \in I$. If there exists at least one such $i$, we simply say the reactions \textit{share species}.
    
\end{define}
\begin{define}
    The \textit{reaction graph} (or \textit{R-graph}) is a graph where nodes correspond to reaction vectors in the network. An edge is drawn between two nodes corresponding to reactions $R_i$ and $R_j$ if and only if they share species. An edge is positive if all shared species have the same sign and negative if all shared species have opposite signs. Otherwise, the edge has no sign.
\end{define}

\begin{define}
    The \textit{$\mathcal{RI}$-graph} is a directed graph where nodes represent reactions in a reaction network. A directed edge from $\Gamma_i$ to $\Gamma_j$ is drawn if either:
    \begin{itemize}
        \item $\Gamma_j$ is irreversible and at least one species of $\Gamma_i$ is a reactant of $\Gamma_j$, or
        \item $\Gamma_j$ is reversible and the two reactions share at least one species.
    \end{itemize}
\end{define}

 In the context of R-graphs, we use ``reaction" and ``node" interchangeably. We assume that all R-graphs are connected. If an R-graph is disconnected, species in different components evolve independently. In other words, if species \( i \) and \( j \) belong to different components, then

\[
\left( \frac{\partial\, (\Gamma R)}{\partial x_j} \right)_i = 0 \quad \text{and} \quad \left( \frac{\partial\, (\Gamma R)}{\partial x_i} \right)_j = 0.
\]

Our results apply separately to each component—for example, if each component has a decreasing norm, the maximum norm across all components also decreases.

 In \cite{Angeli2010-hj}, edges are defined with opposite sign conventions: a positive edge links reactions sharing a species with opposite signs, while a negative edge links reactions sharing a species with the same sign. We deviate from this by reversing the signs and disallowing multiple edges between reactions.

\subsection{Reaction Network Kinetics}

A reaction network induces a system of ordinary differential equations: 
\begin{equation}\label{eq_main}
    \dot{x} = \Gamma R(x),
\end{equation}
where \( x \in \mathbb{R}^n_{\geq 0} \) represents species concentrations, and \( R(x) \) is a column vector of reaction rates. The \( i \)th entry \( R_i(x) \) is a function \( R_i: \mathbb{R}^n_{\geq 0} \to \mathbb{R} \) describing the rate of the \( i \)th reaction, corresponding to the \( i \)th column \( \Gamma_i \) of \( \Gamma \). 

We classify species as \textit{reactants} if they have negative entries in \( \Gamma_i \) and \textit{products} if they have positive entries. A set of species \( I \) is \textit{present} if \( (x)_I > 0 \) entrywise.

Each reversible reaction rate function \( R_i(x) \) satisfies:

\begin{enumerate}
    \item \( R_i(x) \) is continuously differentiable (\( C^1 \)).
    \item Reaction rates obey the conditions:
    \[
    R_i(x) 
    \begin{cases}
        \geq 0, & \text{if not all products are present,} \\
        \leq 0, & \text{if not all reactants are present.}
    \end{cases}
    \]
    \item Partial derivatives satisfy:
    \[
    \frac{\partial R_i(x)}{\partial x_j} 
    \begin{cases}
        \geq 0, & \text{if \( x_j \) is a reactant,} \\
        \leq 0, & \text{if \( x_j \) is a product,} \\
        = 0,  & \text{otherwise.}
    \end{cases}
    \]
    \item Strict partial derivatives hold when all relevant species are present:
    \[
    \frac{\partial R_i(x)}{\partial x_j} 
    \begin{cases}
        > 0, & \text{if \( x_j \) is a reactant and all reactants are present,} \\
        < 0, & \text{if \( x_j \) is a product and all products are present.}
    \end{cases}
    \]
\end{enumerate}

For irreversible reactions, we impose:
\begin{enumerate}
    \item \( R_i \) is \( C^1 \).
    \item \( R_i = 0 \) if any reactant concentration \( x_i = 0 \).
    \item \( \frac{\partial R_i}{\partial x_j} \geq 0 \), with \( = 0 \) if \( x_j \) is not a reactant.
    \item \( \frac{\partial R_i}{\partial x_j} > 0 \) if \( x_j \) is a reactant and all reactants are present.
\end{enumerate}

These kinetic assumptions are satisfied by several commonly used kinetics, such as mass-action, Michaelis–Menten, and Hill kinetics. For additional explanations, examples, and background on reaction networks, see \cite{ANGELI2008128}.

These conditions define the class of differential equations we consider. Throughout, we assume that all systems of the form \( \dot{x} = f(x) \) satisfy \( f \in C^1 \) on \( \mathbb{R}^n_{\geq 0} \). By \( C^1 \) on a closed set, we mean that \( f \) admits a \( C^1 \) extension to an open neighborhood of that set.

Since the vector field \( f \) is \( C^1 \), it follows that \( f \) is locally Lipschitz. Therefore, solutions exist, are unique, and the flow \( \phi_t \) depends continuously on initial conditions (see, for example, \cite[Theorem~54]{sontag2013mathematical} and \cite[Theorem~1]{sontag2013mathematical}).

We assume that for any system of ordinary differential equations \( \dot{x} = f(x) \), the evolution operator \( \phi_t \) exists and is well defined for all \( t \geq 0 \).

We assume that every species participates in at least one reaction—either as part of a reversible reaction or as a reactant in an irreversible one. This ensures that no species is completely disconnected from the dynamics of the reaction network. Later, we introduce dummy species, but we assume that the initial reaction network always satisfies this condition.

\subsection{Monotonicity and Contractivity}

A dynamical system is \textit{monotone} if its evolution preserves a partial ordering on its state space. More formally, consider a system with invariant state space $\mathbb{X} \subseteq \mathbb{R}^n$ (i.e., $\mathbb{X}$ is forward-invariant) and evolution operator $\phi_t:\mathbb{X} \to \mathbb{X}$, defined for $t \geq 0$. Given a partial order $\leq$ on $\mathbb{X}$, the system is \textit{monotone} if for all $x, y \in \mathbb{X}$,

\[
x \leq y \quad \Rightarrow \quad \phi_t(x) \leq \phi_t(y) \quad \forall t \geq 0.
\]

We focus on partial orders induced by cones. Suppose $\mathbb{A} \subseteq \mathbb{X}$ is an affine subspace and $K$ is a pointed cone such that there exists $j \in \mathbb{X}$ with $j + K \subseteq \mathbb{A}$. Then we define a partial order on $\mathbb{A}$ by

\[
x \leq y \quad \iff \quad y - x \in K.
\]

A system $\dot{x} = f(x)$ with invariant affine space $\mathbb{A}$ is \textit{strongly monotone} with respect to $K$ if, for any two trajectories $x(t)$ and $y(t)$, there exists $T > 0$ such that

\[
y(0) - x(0) \in \partial K \quad \Rightarrow \quad y(t) - x(t) \in \operatorname{Int}(K) \quad \forall t \in (0, T).
\]

That is, trajectories initially on the boundary of $K$ move into its relative interior. Lastly, we say a reaction network is monotone if the system $\dot{x} = \Gamma R(x)$ is monotone for all functions $R(x)$ satisfying the kinetic assumptions.

\begin{define}
A system $\dot{x} = f(x)$ is \textit{weakly contractive} with respect to a norm $\| \cdot \|$ if for any two trajectories $x(t)$ and $y(t)$ with $x(0) \neq y(0)$, we have

\[
\|x(t_2) - y(t_2) \| < \|x(t_1) - y(t_1) \| \quad \forall t_2 > t_1.
\]

The system is \textit{non-expansive} with respect to the norm if

\[
\|x(t_2) - y(t_2) \| \leq \|x(t_1) - y(t_1) \| \quad \forall t_2 > t_1.
\]
\end{define}

We say that a reaction network is weakly contractive when, for every admissible choice of kinetics, the dynamics on the relative interior of each stoichiometric compatibility class are weakly contractive with respect to some norm.

\subsection{Connecting Monotonicity and Contractivity}
\label{sec:connect_montone_contractive}

We now establish a connection between strong monotonicity and weak contractivity. To this end, we introduce the lift of a system. Then, using a theorem we will prove, we can deduce weak contractivity of the original system.

\begin{define}
Given a system \( \dot{x} = f(x) \) in \( \mathbb{R}^n \), the \textit{lifted system} (or lift) is the system \( \dot{z} = g(z) \) in \( \mathbb{R}^{n+1} \), where
\[
z = \begin{bmatrix} x \\ x_{n+1} \end{bmatrix}, \quad g(z) = \begin{bmatrix} f(x) \\ 0 \end{bmatrix}.
\]
\end{define}

\begin{define}
The \textit{lift} of a vector \( v \in \mathbb{R}^n \) is the vector \( v' \in \mathbb{R}^{n+1} \) obtained by appending a 1:
\[
v' = \begin{bmatrix} v \\ 1 \end{bmatrix}.
\]
The \textit{lift} of a set of vectors \( \{v_i\} \subset \mathbb{R}^n \) is the cone generated by their lifts \( \{v_i'\} \subset \mathbb{R}^{n+1} \).
\end{define}

\begin{define}
A set \( \beta \subseteq \mathbb{R}^n \) is:
\begin{itemize}
    \item \textit{Absorbing} if for every \( x \), there exists \( r > 0 \) such that \( x \in r \beta \).
    \item \textit{Symmetric} if \( x \in \beta \Rightarrow -x \in \beta \).
    \item \textit{Good} if it is symmetric, compact, convex, and absorbing.
\end{itemize}
\end{define}

\begin{lem}[Theorem 1.35, \cite{rudin1973functional}]
\label{lem:good set to norm}
   If $\beta$ is a good set, then the function
   \[
   \|x\| = \inf\{t>0 \mid x \in t \beta \}
   \]
   defines a norm on $\mathbb{R}^n$.
\end{lem}

\begin{lem}
\label{lem:transversal cones bounded section}
Let \( V \subset \mathbb{R}^n \) be a linear subspace and $K$ be a closed cone such that \( K \cap V = \{0\} \). If \( x \in \mathbb{R}^n \) is such that \( K \cap (x + V) \neq \emptyset \), then the intersection \( K \cap (x + V) \) is bounded.
\end{lem}

\begin{proof}
Let \( S = K \cap (x + V) \). By assumption, \( S \) is a nonempty convex subset of \( x + V \), which is a closed affine subspace of \( \mathbb{R}^n \). Suppose $k \in S$. We note that since $k \in x + V$ we must in fact have $x + V = k + V$.

Suppose for contradiction that \( S \) is unbounded.  Since $S$ is unbounded,  there exists a sequence \( \{v_i\} \subset S \) such that \( \|v_i - k\| \to \infty \), i.e., \( v_i \notin k + B_i \), where \( B_i \) is the closed Euclidean ball of radius \( i \) centered at the origin.

Fix any \( r > 0 \). Since the segment between \( k \) and \( v_i \) lies entirely in \( S \) and \( \|v_i - k\| \to \infty \), for each \( i \) sufficiently large there exists \( \lambda_i \in (0,1) \) such that:
\[
w_i := \lambda_i v_i + (1 - \lambda_i)k \in \partial B_r + k.
\]
Hence, \( w_i \in S \cap (\partial B_r + x) \).

The set \( \partial B_r + k \) is compact in \( \mathbb{R}^n \), so the sequence \( \{w_i\} \) has a convergent subsequence. Let \( \bar{v} \in \partial B_r + k \) be a limit point. Since \( S \) is the intersection of two closed sets and thus is closed, and all \( w_i \in S \subset k + V \), it follows that \( \bar{v} \in S = K \cap (k + V) \).

As both \( \bar{v} \) and \( k \) lie in \( k + V \), we have \( \bar{v} - k \in V \). Moreover, since the sequence \( \{v_i\} \subset S \) is unbounded, we may replace \( r \) with any positive real number \( \alpha r \) and repeat the construction above. This shows that for all \( \alpha > 0 \), the point
\[
\alpha(\bar{v} - k) + k \in K.
\]
Scaling by $1/\alpha$ we have that \( \bar{v} - k + \frac{k}{\alpha} \in K \) for all $\alpha > 0$. Taking the limit as \( \alpha \to \infty \) and using closedness of \( K \), it follows that \( \bar{v} - k \in K \). Thus \( \bar{v} - k \in K \cap V \), contradicting the assumption that \( K \cap V = \{0\} \).

Therefore, our assumption that \( S \) is unbounded must be false. It follows that \( K \cap (x + V) \) is bounded.
\end{proof}

\setcounter{thm}{1}

\begin{thm}
\label{thm:strongly_monotone_lift_to_weakly contractive_general_2}
Suppose we have a system $\dot{x} = f(x)$ and a linear subspace $S$. Assume the system is forward-invariant on every translation of $S$,
\[
S_a = \mathbb{R}^n_{\geq 0} \cap \{a + S\}.
\]
If the lift of the system is strongly monotone with respect to a cone \( K \) satisfying \( K \cap S = \{0 \} \) and \( \operatorname{span}(K) \supseteq S \) (where $S$ is considered a subset of the lifted space), then the original system is weakly contractive with respect to a norm defined on the relative interior of each \( S_a \).
\end{thm}

\begin{proof}
Consider the lifted system. By the assumption \( \operatorname{span}(K) \supseteq S \), and using Lemma~\ref{lem:transversal cones bounded section}, there exists a vector \( b \in K \) such that the intersection \( P = (b + S) \cap K \) is nonempty, bounded, and satisfies \( \operatorname{AffHul}(P) = b + S \).

Define \( H = P - P \). Let \( B = \operatorname{conv}\{P, 0\} \), and set \( B_p = B - P \). Since both \( B \) and \( P \) are compact, the set \( B_p \) is compact as well.

Let \( x \in \mbox{Int}(S_c) \), where \( S_c \) is an arbitrary translation of \( S \). Choose \( \epsilon > 0 \) small enough so that \( x + \epsilon B_p \subseteq \operatorname{Int}(\mathbb{R}^n_{\geq 0}) \). For any \( y \in x + \epsilon H \), write \( y - x = \epsilon(p_1 - p_2) \) for some \( p_1, p_2 \in P \), so \( x = z + \epsilon p_2 \), \( y = z + \epsilon p_1 \), where \( z = x - \epsilon p_2 \). Since \( P \subset K \), both \( x \) and \( y \) lie in the cone \( z + K \), and in particular,
\[
\{x, y\} \subset (z + K) \cap S_c.
\]

Observe that \(z \in  x + \epsilon B_p \), which is contained in \( \operatorname{Int}(\mathbb{R}^n_{\geq 0})\) by construction. Since \(x-z=\varepsilon p_2\in K\) and \(y-z=\varepsilon p_1\in K\), monotonicity of the flow gives
\[
  \phi_t(z)\le\phi_t(x)\quad\text{and}\quad
  \phi_t(z)\le\phi_t(y)\qquad(t>0).
\]

Because \( S_c \) is forward-invariant, it follows that
\[
(\phi_t(z) + K) \cap S_c = \phi_t(z) - z + ((z + K) \cap S_c) \supseteq \phi_t(z) - z + \{\phi_t(x), \phi_t(y)\}.
\]
Therefore, by strong monotonicity, 
\[
\{\phi_t(x), \phi_t(y)\} \subset \operatorname{Int}(\phi_t(z) - z + (x + \epsilon(P - p_2))),
\]
and since this is a translated copy of a scaled version of \( H \), we conclude that
\[
\phi_t(y) \in \operatorname{Int}(\phi_t(x) + \epsilon' H)
\]
for some \( \epsilon' < \epsilon \). Taking \( H \) as the unit ball of a norm, this implies local weak contractivity in a neighborhood of \( x \).

To extend this to all of \( \operatorname{Int}(S_b) \), take any \( x, y \in \operatorname{Int}(S_b) \). The segment from \( x \) to \( y \) is compact, and can be covered by finitely many open sets where local weak contractivity holds. Selecting partition points along the segment within these sets, we obtain a finite sequence
\[
x = x_0, x_1, \dots, x_n = y
\]
such that each pair \( (x_{i-1}, x_i) \) lies within a local weak contractivity neighborhood. For each \( i \), there exists \( t > 0 \) such that
\[
\|\phi_t(x_i) - \phi_t(x_{i-1})\| < \|x_i - x_{i-1}\|.
\]
Summing over \( i \) and using the triangle inequality, we conclude:
\[
\|\phi_t(y) - \phi_t(x)\| \leq \sum_{i=1}^n \|\phi_t(x_i) - \phi_t(x_{i-1})\| < \sum_{i=1}^n \|x_i - x_{i-1}\| = \|y - x\|.
\]
Thus, the system is weakly contractive on \( \operatorname{Int}(S_b) \), as claimed.  
\end{proof}

The key idea of this proof is that we can choose $a$ such that $K \cap \operatorname{AffHul}(S_a) = P$ is a nonempty set, allowing $P - P$ to serve as the unit ball of a norm under which the system $\dot{x} = f(x)$ is weakly contractive on each $S_a$. 

This result is closely related to the construction of a Finsler structure in \cite{Mierczyński1991}, where the author shows that, under mild technical conditions, a cooperative system with a first integral admits a Finsler metric that renders it non-expansive. While our approach uses similar methods, the setting differs: the systems we consider may be monotone with respect to a cone other than the positive orthant, and it is invariant on affine spaces rather than level sets of a function.

\section{Conditions for Monotonicity and Strong Monotonicity}

\subsection{Preliminaries}

\label{sec:Conditions for monotone}
To determine whether a network is monotone we will rely on Proposition 1.5 in \cite{WALCHER2001543}. Let $\mathcal{J}_f(x)$ be the Jacobian of a $C^1$ function $f(x)$ evaluated at a point $x$. We have

\begin{thm}
\label{thm:monotone}
    \cite{WALCHER2001543} A system $\dot{x} = f(x)$ with an open and forward invariant state space $\mathbb{X}$ is monotone with respect to a proper, pointed and convex cone $K$ iff for all $x \in \mathbb{X}$ and for all $k_1\in \partial K$ and $k_2 \in K^*$ such that $\langle k_1, k_2 \rangle = 0$ we have that $\langle \mathcal{J}_f(x) k_1, k_2 \rangle \geq 0$.
\end{thm}

By restricting to an invariant subspace, we can relax the requirement that the cone be proper. If a system $\dot{x} = f(x)$ has an invariant subspace $S$, then we define the restricted system to have its state space be $S$, and every point in $S$ is evolving in time according to $\dot{x} = f(x)$. Thus we will only need the cone to be proper in an invariant subspace. In particular, we have:

\begin{cor}
Suppose a system has an invariant linear subspace \( S \subset \mathbb{R}^n \), and that for all \( j \in \mathbb{R}^n \), the affine space \( S + j \) is also invariant. Let \( K \subset S \) be a cone with the same dimension as \( S \). Suppose further that for all \( x \in S + j \), and for all \( k_1 \in K \), \( k_2 \in K^* \) satisfying \( \langle k_1, k_2 \rangle = 0 \), we have
\[
\langle \mathcal{J}_f(x) k_1, k_2 \rangle \geq 0.
\]
Then the system restricted to \( S + j \) is monotone with respect to \( K \). In particular, the system is monotone when restricted to \( S \).
\end{cor}

\begin{proof}
Since \( \operatorname{Im}(\mathcal{J}_f(x)) \subset S \), we have
\[
\langle \mathcal{J}_f(x) k_1, k_2 \rangle = \langle \mathcal{J}_f(x) k_1, \pi(k_2) \rangle,
\]
where \( \pi(k_2) \) denotes the orthogonal projection of \( k_2 \) onto \( S \). Thus, it suffices to consider \( k_2 \in S \cap K^* \) instead of \( k_2 \in K^* \). By Theorem~\ref{thm:monotone}, the system is monotone on \( S + j \), and in particular on \( S \).
\end{proof}

Note that this argument also applies to the relative interior of any sets of the form $\mathbb{R}^n_{\geq 0} \cap (S + j)$. The state spaces of the reaction networks we consider will be of this form.

In the sequel we will write $\mathcal{J}_f$ instead of $\mathcal{J}_f(x)$, where it is to be understood that statements involving $\mathcal{J}_f$ are to hold for all $x$ in our system's state space.  We can manipulate the expression in Theorem \ref{thm:monotone} to get a more geometric statement, specifically for reaction networks. First, we will define a few regions:

We will need a criterion for determining strong monotonicity.

\begin{thm}[{\cite[Theorem~3.6]{HIRSCH2006239}}]
\label{thm:strong monotone strengthened}
Suppose we have a system $\dot{x} = f(x)$ with an open state space and a proper, pointed cone $K$. Suppose also that
\[
\langle \mathcal{J}_f k_1, k_2 \rangle \geq 0
\]
for all $k_1 \in \partial K$ and $k_2 \in K^*$ such that $\langle k_1,k_2 \rangle = 0$. If for each $k_1 \in \partial K$, there exists a $k_2 \in K^*$ such that $\langle k_1,k_2 \rangle = 0$ and
\[
\langle \mathcal{J}_f k_1, k_2 \rangle > 0,
\]
then the system is strongly monotone with respect to $K$.
\end{thm}

\subsection{Reaction Network Results}

\begin{define}
\label{def:Q regions}
    Given a reaction vector $v \in \mathbb{R}^n$ corresponding to a reversible reaction, define:
    \[
    Q_1(v) = \{x \in \mathbb{R}^n \mid (x)_i(v)_i \geq 0 \text{ for all } 1 \leq i \leq n \}.
    \]
    For a reaction vector $v$ corresponding to an irreversible reaction, with the reactant species indexed by $I$, define:
    \[
    Q_1(\Gamma_1) = \{x \in \mathbb{R}^n| (x)_i (v)_i \geq 0 \text{ for all } i \in I \}
    \]
    
    Additionally, define:
    \begin{align*}
    Q_1^+(v) &= Q_1(v) \setminus Q_1(-v), \\
    Q_1^-(v) &= Q_1(-v) \setminus Q_1(v), \\
    Q_2(v) &= Q_1(v) \cap Q_1(-v).
    \end{align*}
\end{define}

For example, take the vector $v = [1,2,0,-1,-2,0]$. Then we have that $Q_1(v)$ consists of all vectors with sign pattern $[\geq,\geq,*,\leq,\leq,*]$ (here $\geq$ indicates nonnegative and $\leq$ nonpositive, and $*$ can be anything). We have that $Q_2(v)$ consists of all vectors with sign pattern $[0,0,*,0,0,*]$. Lastly, we have that $Q_1^+(v)$ contains vectors of the form $[\geq,\geq,*,\leq,\leq,*]$, but at least one of the non-star entries must be nonzero, such as $[1,0,0,0,0,0]$ but not $[0,0,1,0,0,0]$.

Notably, the following properties hold in general:
\[
Q_1^+(-v) = Q_1^-(v), \quad Q_2(v) = Q_2(-v), \quad Q_1(v) \cup Q_1(-v) = Q_1^-(v) \cup Q_2(v) \cup Q_1^+(v).
\]

For the rest of this section, we assume any cones we consider are closed, convex, pointed, and polyhedral. Recall that the dual face is defined in section \ref{sec:background}.

\begin{lem}
\label{lem:only need consider extreme rays}
Let \( K \) be a cone generated by extreme ray vectors \( \{k_i\}_{i=1}^n \). Suppose \( k \) is a convex combination of a subset of these extreme rays, i.e.,  
\[
k = \sum_j \lambda_j k_j, \quad \text{where } \lambda_j > 0, \sum_j \lambda_j = 1.
\]
Then, the dual face satisfies \( k^* \subseteq k_j^* \) for each \( k_j \) in the combination.
\end{lem}

\begin{proof}
By definition, the dual face of \( k \) is
\[
k^* = \{ k' \in K^* \mid \langle k, k' \rangle = 0 \}.
\]
Substituting \( k = \sum_j \lambda_j k_j \), we obtain:
\[
0 = \langle k, k' \rangle = \sum_j \lambda_j \langle k_j, k' \rangle.
\]
Since each \( \lambda_j > 0 \), it follows that \( \langle k_j, k' \rangle = 0 \) for all \( j \). Thus, \( k' \in k_j^* \) for each contributing \( k_j \), implying \( k^* \subseteq k_j^* \).
\end{proof}

\begin{lem}
\label{lem:signs on dual faces}
Consider a reaction network with a single reaction vector $\Gamma_1$ and the system $\dot{x} = \Gamma_1 R(x)$. Then the system is monotone with respect to a cone $K$ if and only if the following conditions hold:

\begin{enumerate}
    \item For all extreme rays $k \in K$, we have $k \in Q_1(\Gamma_1) \cup Q_1(-\Gamma_1)$.
    \item If extreme ray $k \in Q_1^+(\Gamma_1)$, then for all $k' \in k^*$, we have $\langle \Gamma_1,k' \rangle \leq 0$.
    \item If extreme ray $k \in Q_1^-(\Gamma_1)$, then for all $k' \in k^*$, we have $\langle \Gamma_1,k' \rangle \geq 0$.
\end{enumerate}
\end{lem}

\begin{proof}
Suppose the system is monotone with respect to \( K \). Since there is only one reaction, we have
\[
\mathcal{J}_f = \Gamma_1 \partial R_1,
\]
where \( \partial R_1 = [\partial_1 R_1, \partial_2 R_1, \dots, \partial_n R_1] \). By Theorem \ref{thm:monotone}, monotonicity holds if and only if
\[
\langle \mathcal{J}_f k_1, k_2 \rangle \geq 0 \quad \forall k_2 \in k_1^*.
\]
Expanding this condition:
\[
\langle \mathcal{J}_f k_1, k_2 \rangle = \langle \Gamma_1 \partial R_1 k_1, k_2 \rangle = \langle \Gamma_1, k_2 \rangle \langle \partial R_1, k_1\rangle.
\]

If \( k_1 \notin Q_1(\Gamma_1) \cup Q_1(-\Gamma_1) \), we can choose reaction rates such that \( \langle \partial R_1, k_1 \rangle \) takes any sign, forcing \( \langle \Gamma_1, k_2 \rangle = 0 \) for all \( k_2 \in k_1^* \). Since \( k_1^* \) is an \( (n-1) \)-dimensional face, this implies \( \Gamma_1 \) is perpendicular to \( k_1^* \), making it a multiple of \( k_1 \), contradicting the assumption. Thus, condition (1) must hold.

For condition (2), if \( k_1 \in Q_1^+(\Gamma_1) \), then \( \langle \partial R_1, k_1 \rangle < 0 \), requiring \( \langle \Gamma_1, k_2 \rangle \leq 0 \) to satisfy the monotonicity condition. Similarly, for \( k_1 \in Q_1^-(\Gamma_1) \), we must have \( \langle \Gamma_1, k_2 \rangle \geq 0 \), establishing condition (3).

Now, suppose the three conditions hold. For an extreme ray \( k_1 \), we already verified that \( \langle \mathcal{J}_f k_1, k_2 \rangle = \langle \Gamma_1, k_2 \rangle \langle \partial R_1, k_1\rangle \geq 0 \). If \( k_1 \) is not extreme ray, express it as a convex combination \( k_1 = \sum_i \lambda_i l_i \) for extreme rays \( l_i \). By Lemma \ref{lem:only need consider extreme rays}, any \( k_2 \in k_1^* \) is also in \( l_i^* \) for all \( i \), and we have already established \( \langle \mathcal{J}_f l_i, k_2 \rangle \geq 0 \) for $k_2 \in l_i^*$. Thus, 
\[
\langle \mathcal{J}_f k_1, k_2 \rangle = \sum_i \lambda_i \langle \mathcal{J}_f l_i, k_2 \rangle \geq 0.
\]
By Theorem \ref{thm:monotone}, the system is monotone with respect to \( K \).
\end{proof}

\begin{cor}
\label{lem:reac_vec_extreme ray}
Let a reaction network consist of a single reaction vector \( \Gamma_1 \), which is an extreme ray of a cone \( K \) or \( -K \). Assume all extreme rays \( k \) of \( K \) satisfy \( k \in Q_1(\Gamma_1) \cup Q_1(-\Gamma_1) \). Then the network is monotone with respect to \( K \) if \( \Gamma_1 \) (respectively, \( -\Gamma_1 \)) is the unique extreme ray in \( Q_1^+(\Gamma_1) \) (respectively, \( Q_1^-(\Gamma_1) \)).
\end{cor}

\begin{proof}
Suppose \( \Gamma_1 \) is a positive multiple of an extreme ray of \( K \), so \( \Gamma_1 \in K \). Then, for all \( k' \in K^* \), we have:
\[
\langle \Gamma_1, k' \rangle \geq 0,
\]
so condition (3) of Lemma~\ref{lem:signs on dual faces} is satisfied.

Since \( \Gamma_1 \) is the only extreme ray of \( K \) lying in \( Q_1^+(\Gamma_1) \), condition (2) applies only to \( k = \Gamma_1 \). In this case, \( k^* \subseteq K^* \), and the dual pairing yields:
\[
\langle \Gamma_1, k' \rangle = 0 \quad \text{for all } k' \in k^*,
\]
so condition (2) is also satisfied.

Condition (1) holds by hypothesis. Therefore, all conditions of Lemma~\ref{lem:signs on dual faces} are satisfied, and the system is monotone with respect to \( K \).

The case where \( -\Gamma_1 \) is the extreme ray of \( K \) is handled identically.
\end{proof}

\begin{define}
    We say a face $F$ of a cone is a \textit{mixed face for $\Gamma_i$} (or just a \textit{mixed face} if it is clear from the context what $\Gamma_i$ is) if $F$ contains extreme ray vectors from both $Q_1^+(\Gamma_i)$ and $Q_1^-(\Gamma_i)$.
\end{define}

\begin{lem}
\label{lem:contained in mixed faces}
    Suppose that $\Gamma_i$ is monotone with respect to a cone $K$. If $F \subseteq K$ is a mixed face for $\Gamma_i$, then $\Gamma_i \in \spans(F)$.
\end{lem}

\begin{proof}
    Suppose that $k \in Q^+_1(\Gamma_i) \cap F$ and $k' \in Q^-_1(\Gamma_i) \cap F$, where $k,k'$ are both extreme rays. If $k_2 \in F^*$ then this implies that $k_2 \in k^* \cap k'^*$. Lemma \ref{lem:signs on dual faces} then implies that for all $k_2 \in F^*$ that both $\langle \Gamma_i, k_2 \rangle \geq 0$ and $\langle \Gamma_i, k_2 \rangle \leq 0$. This implies that $\langle \Gamma_i, k_2 \rangle = 0$ for all $k_2 \in F^*$. Thus we must have $\Gamma_i$ in the space perpendicular to $F^*$, or $\Gamma_i \in \spans(F)$.
\end{proof}

\begin{lem}
\label{lem:one_reac_makes_strict_monotone}
Let $\mathcal{N}$ be a reaction network. Then:
\begin{enumerate}
    \item The network $\mathcal{N}$ is monotone with respect to a cone $K$ if and only if each individual reaction is monotone with respect to $K$.
    
    \item Suppose $\mathcal{N}$ is monotone. Then, under the same assumptions on $k_1$ and $k_2$ as in Theorem~\ref{thm:strong monotone strengthened}, we have
    \[
    \langle \mathcal{J}_f k_1, k_2 \rangle > 0
    \quad \text{if and only if} \quad
    \langle \mathcal{J}_{f^*} k_1, k_2 \rangle > 0
    \]
    for some single-reaction subsystem $f^*$ corresponding to a reaction vector $\Gamma_i$ from $\mathcal{N}$.
\end{enumerate}
\end{lem}

\begin{proof}
    Consider the first statement. Suppose our network $\mathcal{N}$ consists of reaction vectors $\{\Gamma_i\}_{i=1}^n$. We have that $\mathcal{J}_f k_1 = \sum_{i=1}^n \Gamma_i \partial R_i$. Clearly $\langle \mathcal{J}_f k_1, k_2 \rangle = \langle  \sum_{i=1}^n \Gamma_i \partial R_i k_1, k_2 \rangle \geq 0$ is true if and only if it is true for each individual summand, i.e., reaction. If a system consisting of one reaction vector is not monotone, we could set the other reaction rate derivatives to arbitrarily small values, to get that $\langle  \sum_{i=1}^n \Gamma_i \partial R_i k_1, k_2 \rangle < 0 $; a contradiction. If all the reactions are monotone with respect to a cone $K$, then we simply have a sum of nonnegative terms, which must be nonnegative.

    Consider the second statement. Note if $\langle  \sum_{i=1}^n \Gamma_i \partial R_i k_1, k_2 \rangle = \sum_{i=1}^n \langle   \Gamma_i \partial R_i k_1, k_2 \rangle > 0 $ then the inequality must be true for at least one of the summands (in our sum of nonnegative terms). If one of the summands is strictly positive for each of our possible choices of $k_1,k_2$, then we would have $\langle  \sum_{i=1}^n \Gamma_i \partial R_i k_1, k_2 \rangle > 0 $ (since, due to monotonicity, each summand is nonnegative) and can apply Theorem \ref{thm:strong monotone strengthened}.
\end{proof}

\begin{lem}
\label{lem:subtract_off}
    Suppose we are given a polyhedral cone $K$ and a reaction network consisting of a single reaction vector $\Gamma_1$. Suppose that the following properties hold: 
    \begin{enumerate}
        \item $k \in  Q_1(\Gamma_1) \cup Q_1(-\Gamma_1)$ for all extreme rays $k \in K$.
        \item For all extreme rays $k \in Q_1^+(\Gamma_1)$ we can find a vector $k' \in Q_1(-\Gamma_1) \cap K$ such that $\Gamma_1 = k - k'$.
        \item For all extreme rays $k \in Q_1^-(\Gamma_1)$ we can find a vector $k' \in Q_1(\Gamma_1) \cap K$ such that $\Gamma_1 = k' - k$.
    \end{enumerate} 
    
    Then our network is monotone with respect to $K$.
\end{lem}

\begin{proof}
    Take an arbitrary extreme ray $k \in Q_1^+(\Gamma_1) \cap K$, and find $k'$ such that $\Gamma_1 = k - k' $ where $k' \in Q_1(-\Gamma_1) \cap K$. Now for $k_2 \in k^*$ we have that $ \langle \Gamma_1, k_2 \rangle = \langle -k', k_2 \rangle \leq 0$ (since in general $\langle a,b \rangle \geq 0$ for $a$ in a cone and $b$ in its dual cone). A similar argument applies for $k \in Q_1^-(\Gamma_1)$. Thus our reaction vector $\Gamma_1$ satisfies the condition of Lemma \ref{lem:signs on dual faces}, and so our reaction network is monotone with respect to $K$.

\end{proof}

The following lemma says that a point in the relative interior of the convex hull of a set must have positive barycentric coordinates. This is a standard result, but we include a proof for convenience of the reader.

\begin{lem}
\label{lem:convex_sums_for_interior}
    Suppose we have a polytope $P$ with a set of vertices $V = \{k_j| 1 \leq j \leq n\}$. For each $k_i \in V$ and each $k$ in the interior of $P$ we can write $k = \sum_{j=1}^n \alpha_j k_j$ where $\sum_{j=1}^n \alpha_j = 1$, $\alpha_j \geq 0$ for all $j$, and $\alpha_i \neq 0$. 
\end{lem}

\begin{proof}
    Select an arbitrary vertex $k_i$ and arbitrary $k$ in the interior of $P$. Consider a line between these two points, and note that this line must leave our convex figure at precisely two points, one point being $k_i$. Label the point where the line exits the polytope (which is not $k_i$) as $k_e$. Then $k$ is a convex combination of $k_i$ and $k_e$, and $k_e$ is a convex combination of some of the vertices, and so we are done.
\end{proof}

\begin{lem}
\label{lem:subtract_off_strong_2}
Let \( \mathcal{N} \) be a set of reaction vectors, and let \( K \subset \mathbb{R}^n \) be a cone such that for each \( \Gamma_i \in \mathcal{N} \), the system \( \dot{x} = \Gamma_i R_i(x) \) satisfies the conditions of Lemma~\ref{lem:subtract_off}. Suppose further that for every proper face \( F \subset K \), there exist a reaction vector \( \Gamma_i \in \mathcal{N} \), an extreme ray \( k \in F \), and a vector \( k_i \in K \setminus F \) such that:
\begin{enumerate}
    \item \( k \notin Q_2(\Gamma_i) \),
    \item There exists \( j \in \{0,1\} \) such that \( \Gamma_i = (-1)^j(k - k_i) \).
\end{enumerate}
Then the system \( \dot{x} = \sum_i \Gamma_i R_i(x) \) is strongly monotone with respect to the cone \( K \).
\end{lem}

\begin{proof}
Let \( f = \sum_i \Gamma_i R_i(x) \) be the vector field associated with the network \( \mathcal{N} \). Fix any nonzero \( k' \in \partial K \), and let \( F \subset K \) be the minimal face such that \( k' \in \operatorname{Int}(F) \). Since \( F \) is a proper face, by assumption there exists a reaction \( \Gamma_i \in \mathcal{N} \), an extreme ray \( k \in F \), and a vector \( k_i \in K \setminus F \) such that \( k \notin Q_2(\Gamma_i) \) and \( \Gamma_i = k - k_i \) or \( \Gamma_i = k_i - k \).

We claim there exists a vector \( k_2 \in F^* \subset K^* \) such that
\[
\langle \Gamma_i, k_2 \rangle = \langle -k_i, k_2 \rangle \neq 0.
\]
To see this, suppose by contradiction that \( \langle k_2, k_i \rangle = 0 \) for all \( k_2 \in F^* \). Then \( k_i \in (F^*)^\perp \). Since \( F^* \subset K^* \), and \( (F^*)^\perp \cap K = F \), it follows that \( k_i \in F \), contradicting the assumption that \( k_i \notin F \). Hence, such a \( k_2 \in F^* \) with \( \langle \Gamma_i, k_2 \rangle \neq 0 \) must exist.

Next, since \( k' \in \operatorname{relint}(F) \), Lemma~\ref{lem:convex_sums_for_interior} implies that \( k' \) can be written as a convex combination of the extreme rays of \( F \), including \( k \). Moreover, since \( k \notin Q_2(\Gamma_i) \), we must have \( k \in Q_1^+(\Gamma_i) \) or \( k \in Q_1^-(\Gamma_i) \). Without loss of generality, assume \( k \in Q_1^+(\Gamma_i) \).

Suppose, toward a contradiction, that \( \langle \partial R_i, k' \rangle = 0 \). Since \( k' \) is a convex combination of rays in \( F \), this would imply that the positive and negative contributions to \( \langle \partial R_i, k' \rangle \) cancel out. In particular, some other extreme rays in \( F \) must lie in \( Q_1^-(\Gamma_i) \), implying that \( F \) is a mixed face for \( \Gamma_i \). By Lemma~\ref{lem:contained in mixed faces}, this would imply \( \Gamma_i \in \operatorname{span}(F) \), and since \( k = \Gamma_i + k_i \) or \( k_i = \Gamma_i + k \), it would follow that \( k_i \in F \), again contradicting \( k_i \notin F \). Therefore,
\[
\langle \partial R_i, k' \rangle \neq 0.
\]

Now compute:
\[
\langle \mathcal{J}_f k', k_2 \rangle \geq \langle \mathcal{J}_{\Gamma_i R_i} k', k_2 \rangle = \langle \Gamma_i, k_2 \rangle \langle \partial R_i, k' \rangle > 0.
\]
Here we used Lemma~\ref{lem:subtract_off}, which guarantees that the two inner product terms have the same sign, and we have shown both are nonzero.

Since this holds for all nonzero \( k' \in \partial K \), by Theorem~\ref{thm:strong monotone strengthened}, the system is strongly monotone with respect to \( K \).
\end{proof}

\section{Dynamical Conclusions on Global Convergence}

Our main result establishes weak contractivity only on the relative interior of each stoichiometric compatibility class. As such, additional arguments are needed to conclude convergence to a unique equilibrium on the full class.

A key concept in the theory of reaction networks is \textit{persistence}, which ensures that trajectories starting in the relative interior of the nonnegative orthant do not asymptotically approach the boundary. More precisely, a reaction network is said to be persistent if it has compact stoichiometric compatibility classes and
\[
\omega(x) \cap \partial \mathbb{R}_{\geq 0}^n = \emptyset \quad \text{for all } x \in \operatorname{Int}(\mathbb{R}_{\geq 0}^n).
\]

Several results in the literature provide sufficient conditions for persistence. For the remainder of this section, assume the reaction network \( \Gamma \) satisfies the hypotheses of Theorem \ref{thm:collection of theorems}.

Let \( \mathbb{R}^I \subseteq \mathbb{R}^n \) denote the coordinate subspace where species outside of \( I \) are zero. Define \( \mathbb{R}_{\geq 0}^I = \mathbb{R}_{\geq 0}^n \cap \mathbb{R}^I \).

\begin{lem}
\label{lem:nonexpansivity closed condition}
Let $\dot{x} = f(x)$ be a dynamical system with closed state space $X$, and suppose the system is nonexpansive on an open subset $U \subset X$. If the closure $\bar{U}$ is forward-invariant, then the system is nonexpansive on $\bar{U}$.
\end{lem}

\begin{proof}
Assume for contradiction that the system is not nonexpansive on $\bar{U}$. Then there exist \( x, y \in \bar{U} \) and \( t > 0 \) such that
\[
\| \phi_t(y) - \phi_t(x) \| > \| y - x \|.
\]
Let \( \epsilon = \| \phi_t(y) - \phi_t(x) \| - \| y - x \| > 0 \). Since \( x, y \in \bar{U} \), for any \( \delta > 0 \), we can choose \( x', y' \in U \) such that
\[
\| x' - x \| < \delta, \quad \| y' - y \| < \delta.
\]
By continuity of the flow \( \phi_t \), for any \( \eta > 0 \), we can choose \( \delta > 0 \) small enough so that
\[
\| \phi_t(x') - \phi_t(x) \| < \eta, \quad \| \phi_t(y') - \phi_t(y) \| < \eta.
\]
Now we estimate:
\[
\| \phi_t(y') - \phi_t(x') \| 
\geq \| \phi_t(y) - \phi_t(x) \| - \| \phi_t(y) - \phi_t(y') \| - \| \phi_t(x) - \phi_t(x') \| 
> \| y - x \| + \epsilon - 2\eta.
\]
Also,
\[
\| y' - x' \| \leq \| y' - y \| + \| y - x \| + \| x - x' \| < \| y - x \| + 2\delta.
\]
So if we choose \( \delta \) and \( \eta \) such that \( \epsilon > 2\delta + 2\eta \), then
\[
\| \phi_t(y') - \phi_t(x') \| > \| y' - x' \|,
\]
which contradicts nonexpansivity on \( U \). Therefore, the system must be nonexpansive on \( \bar{U} \).
\end{proof}

We will need Corollary 2 from \cite{10886667}:

\begin{cor}
\label{cor:isometry on attractor}\cite{10886667}
    Suppose we have a $C^1$ system $\dot{x} = f(x)$ with compact forward invariant state space $X$. Then for any real number $t \geq 0$ the time evolution operator $\phi_t$ is an isometry on the set ${\gattract}  = \cap_{t\geq 0 } \phi_t(X)$ (i.e., the global attractor of the system).
\end{cor}

With the same assumptions on the system, we also have the following:

\begin{lem}
\label{lem:omega limit sets in global attractor}\cite{10886667}
    Every $\omega$-limit set is contained in ${\gattract}$.
\end{lem}

\begin{lem}
\label{lem:unique_equilibrium_relative_interior}
If a stoichiometric compatibility class contains an equilibrium in its relative interior, then this equilibrium is unique, and all trajectories—including those on the boundary—converge to it.
\end{lem}

\begin{proof}
Let \( p \) be an equilibrium in the relative interior. For any boundary point \( b \), write:
\[
\|b - p\| = \|b - x\| + \|x - p\|,
\]
where \( x = \lambda b + (1-\lambda) p \), for some $0 < \lambda < 1$, is a point in the relative interior. Since weak contractivity ensures \( \|x - p\| \) is strictly decreasing and $\|b-x\|$ is nonincreasing by Lemma \ref{lem:nonexpansivity closed condition}, it follows that \( \|b - p\| \) must also be strictly decreasing.

Now let \( x \) be an arbitrary point in the stoichiometric compatibility class \( S \), and define \( d = \|x - p\| \). Consider the compact, forward-invariant set
\[
B_d = \{ y \in S \mid \|y - p\| \leq d \}.
\]
By Corollary \ref{cor:isometry on attractor}, the flow \( \phi_t \) acts as an isometry on the global attractor \( \mathcal{A} \). However, we have already observed that all boundary points strictly contract toward \( p \), and weak contractivity holds in the relative interior. Thus, for any \( x \neq p \), the distance \( \| \phi_t(x) - p \| \) must decrease over time, contradicting invariance under isometries unless \( x = p \). Therefore, \( \mathcal{A} = \{p\} \), and by Lemma \ref{lem:omega limit sets in global attractor}, all omega-limit sets lie in \( \mathcal{A} \). Hence, every trajectory converges to \( p \).
\end{proof}

\begin{lem}
\label{lem:equilibria_in_boundary}
If a stoichiometric compatibility class has an equilibrium in \( \mathbb{R}^I \) with \( |I| = n-1 \), then it cannot have equilibria outside \( \mathbb{R}^I \). Moreover, all omega-limit sets of trajectories must lie in \( \mathbb{R}^I \).
\end{lem}

\begin{proof}
Assume, for contradiction, that the stoichiometric class contains two equilibria: one at \( p \in \mathbb{R}^I \) and another at \( p' \notin \mathbb{R}^I \). The line segment connecting \( p \) and \( p' \) must pass through the relative interior of the class. Thus, there exist points \( x, x' \) in the relative interior such that
\[
\|p - p'\| = \|p - x\| + \|x - x'\| + \|x' - p'\|.
\]
Since the system is weakly contractive on the relative interior, the norm \( \|x - x'\| \) must be strictly decreasing under the flow. Therefore, \( \|p - p'\| \) must also decrease, contradicting the assumption that both \( p \) and \( p' \) are equilibria (and hence fixed points under the flow).

Now suppose a trajectory starting in the relative interior has an omega-limit point \( w \notin \mathbb{R}^I \). Then, by the same reasoning, the line segment between \( p \) and \( w \) must pass through the relative interior. Weak contractivity again implies that \( \|p - w\| \) is strictly decreasing over time. However, by Corollary \ref{cor:isometry on attractor} and Lemma \ref{lem:omega limit sets in global attractor}, the distance between any two omega-limit points must remain constant. This contradiction implies \( w \in \mathbb{R}^I \).
\end{proof}

\begin{lem}
\label{lem:persistence_implies_convergence}
If the network is persistent, then every trajectory converges to a unique equilibrium in the relative interior of its stoichiometric compatibility class.
\end{lem}

\begin{proof}
Persistence implies that omega-limit sets do not intersect the boundary of the nonnegative orthant. By Lemma~\ref{lem:equilibria_in_boundary}, there are no equilibria on the boundary. Since persistence also guarantees that stoichiometric compatibility classes are compact, each class contains at least one equilibrium. Furthermore, weak contractivity ensures that there is at most one equilibrium in the relative interior of each stoichiometric class. Hence, there exists a unique equilibrium in the relative interior, and every trajectory must converge to it.
\end{proof}

\begin{lem}
\label{lem:reversible_boundary_invariant}
If a reversible reaction network has an equilibrium in the relative interior of \( \mathbb{R}^I \cap \mathbb{R}_{\geq 0}^n \), then \( \mathbb{R}^I \cap \mathbb{R}_{\geq 0}^n \) is forward-invariant.
\end{lem}

\begin{proof}
Let \( p \in \operatorname{Int}(\mathbb{R}^I \cap \mathbb{R}_{\geq 0}^n) \) be an equilibrium, so \( \dot{x} = \Gamma R(x) = 0 \) at \( x = p \). For any index \( k \notin I \), we have \( (p)_k = 0 \) and hence \( (\dot{x})_k = 0 \).

Now consider a reversible reaction \( R_j \) for which species \( k \notin I \) appears as a product. By the kinetic assumptions, \( R_j(p) \geq 0 \). Since \( p \) is an equilibrium, \( R_j(p) = 0 \). If all reactants of \( R_j \) were present at \( p \), then by strict positivity of partial derivatives (assumption 4 of the kinetics), we would have \( R_j(p) > 0 \), a contradiction. Therefore, not all reactants of \( R_j \) are present at \( p \), even though all species in \( I \) are.

It follows that for every \( x \in \mathbb{R}^I \cap \mathbb{R}_{\geq 0}^n \), at least one reactant of \( R_j \) is absent, and hence \( R_j(x) = 0 \). This argument applies to all reactions producing any species \( k \notin I \). Therefore, no species outside \( I \) can be produced from within \( \mathbb{R}^I \cap \mathbb{R}_{\geq 0}^n \), and the dynamics remain confined to this subspace.

Hence, \( \mathbb{R}^I \cap \mathbb{R}_{\geq 0}^n \) is forward-invariant.
\end{proof}

\begin{lem}
\label{lem:reversible_unique_equilibrium}
If the network is reversible, then all trajectories converge to at most one equilibrium, which may lie on the boundary.
\end{lem}

\begin{proof}
If an equilibrium exists in the relative interior, uniqueness follows by Lemma \ref{lem:unique_equilibrium_relative_interior}. Otherwise, suppose there is an equilibrium in the relative interior of \( \mathbb{R}^I \). By Lemma \ref{lem:reversible_boundary_invariant}, the system remains invariant on \( \mathbb{R}^I \).

Since the network is reversible, subsets of reactions still satisfy Theorem \ref{thm:collection of theorems}, ensuring weak contractivity. Suppose there were two equilibria in \( \mathbb{R}^I \) not contained in some \( \mathbb{R}^{I'} \) where \( |I'| = |I| - 1 \). Then, by weak contractivity, they must move toward each other—a contradiction. 

By induction, considering decreasing-dimensional boundary subsets, we conclude that at most one equilibrium exists.
\end{proof}

The following result, known as Yorke’s Fixed Point Theorem (Theorem 1.16 in \cite{FB-CTDS}), ensures the existence of equilibria.

\begin{lem}
\label{lem:fixed_point_compact}
A system with a convex, compact, and forward-invariant set has at least one equilibrium.
\end{lem}

\begin{cor}
\label{cor:global fixed point}
A reaction network with compact stoichiometric compatibility classes has at least one equilibrium in each class. If the network is persistent or reversible, this equilibrium is unique, and all trajectories converge to it.
\end{cor}

\begin{proof}
By Lemma \ref{lem:fixed_point_compact}, at least one equilibrium exists. Uniqueness and global convergence follow from Lemma \ref{lem:persistence_implies_convergence} (for persistent networks) or Lemma \ref{lem:reversible_unique_equilibrium} (for reversible networks).
\end{proof}

\begin{define}
A dynamical system is globally convergent if every point in the state space \( \mathbb{X} \) converges to a unique equilibrium \( p \), i.e., if the conclusion of Corollary \ref{cor:global fixed point} holds.
\end{define}

\section{Checkable Conditions for the Monotonicity of Reaction Networks}
\label{sec:algo for cones}
 We now define several operations that are useful for determining whether a network is monotone with respect to a cone.

\begin{define}
Let \( \Gamma \) be a stoichiometric matrix, and let \( v \in \mathbb{R}^n \) be a vector. A vector \( v' \in \mathbb{R}^n \) is called \textit{permissible} if, for every reaction vector \( \Gamma_i \), we have \( v' \in Q_1(\Gamma_i) \cup Q_1(-\Gamma_i) \).

We say that subtracting (or adding) a reaction vector \( \Gamma_i \) from \( v \) is a \textit{permissible operation} if \( v \notin Q_2(\Gamma_i) \), and the resulting vector \( v - \Gamma_i \) (or \( v + \Gamma_i \)) is permissible.

A cone is said to be \textit{permissible} if it is generated by a set of permissible vectors.

\end{define}

\begin{define}
\label{def:closed}
Let \( A = \{v_1, v_2, \dots, v_k\} \subset \mathbb{R}^n \) be a finite set of vectors. We say \( A \) is \textit{closed} if it satisfies the following conditions:
\begin{enumerate}
    \item For every \( v' \in A \) and every reaction vector \( \Gamma_i \) such that \( v' \notin Q_2(\Gamma_i) \), at least one of the operations \( v' - \Gamma_i \) or \( v' + \Gamma_i \) is permissible. If $v' \in Q_1^+(\Gamma_i)$ then $v' - \Gamma_i \in Q_1(-\Gamma_i)$ and if $v' \in Q_1^-(\Gamma_i)$ then $v' + \Gamma_i \in Q_1(\Gamma_i)$.
    \item All such permissible operations produce vectors that are also contained in \( A \).
\end{enumerate}
\end{define}

\begin{define}
\label{def:viable set}
A nonempty set \( K \) of permissible vectors is called \textit{viable} if it is closed.
\end{define}

We always need our cones to be proper with respect to the invariant subspace we want to draw conclusions about, so we need the following lemma:

\begin{lem}
\label{lem_viable_at_least_stoich_dimension}
    The affine hull of a viable set $M$  always has at least the same dimension as the stoichiometric compatibility classes. The affine hull contains a translation of the affine hull of any stoichiometric compatibility class.
\end{lem}

\begin{proof}
    If $M$ contains the vector $v$, then we can show $\mbox{AffHul}(M) \supseteq v + \mbox{Im}(\Gamma)$. To see this, note that because $M$ is viable, at some point for some $v \in M$ we form either $v - \Gamma_i$ or $v + \Gamma_i$. Assume without loss of generality that we form $v- \Gamma_i$, and so $v' = v - \Gamma_i \in M$. Note that for all $v'' \in \mbox{AffHul}(M)$ and for all $\alpha \in \mathbb{R}$ we must also have $v'' + \alpha(v - v') = v'' + \alpha \Gamma_i \in \mbox{AffHul}(M)$. Since this is true for all $\Gamma_i$ we must have $ v + \mbox{Im}(\Gamma) \subseteq \mbox{AffHul}(M)$.

\end{proof}

\begin{thm}
\label{thm:lift viable for monotone}
     Suppose we have a viable set of vectors. If we define our cone $K$ to be the lift of this set, then our system is monotone with respect to $K$. 
\end{thm}

\begin{proof}
    In this case, we will check that we satisfy the conditions for Lemma 
    \ref{lem:subtract_off} and so that the system is monotone. To check that all the vectors are in $Q_1(\Gamma_1) \cup Q_1(-\Gamma_1)$ we simply need to note that this is part of Definition \ref{def:viable set} for the set to be viable.

    If we have a permissible vector $k \in Q_1^+(\Gamma_i)$, then by the closed condition this implies $k - \Gamma_i \in Q_1(-\Gamma_i)$ is also a member of our set (note that $k + \Gamma_i$ could no longer be a permissible vector, since both vector share at least one species with the same sign). Similar reasoning applies for $k \in Q_1^-(\Gamma_i)$, and so our system must be monotone for each individual reaction, and thus the whole network as well by Lemma \ref{lem:one_reac_makes_strict_monotone}.

\end{proof}

We are also interested in when our systems are strongly monotone.  

\begin{thm}
\label{thm:output_M_to_strong_monotone}
Let $\mathcal{N}$ be a reaction network with stoichiometric matrix $\Gamma$, where every species participates in at least one reaction, and assume the associated $\mathcal{RI}$-graph is strongly connected.

Suppose $M$ is a viable set of permissible vectors such that $\operatorname{AffHul}(M)$ has the same dimension as the stoichiometric compatibility classes. Let $M'$ denote the set of extreme points of the convex hull of $M$, and define the lifted cone $K'$ as the cone generated by the lift of either $M'$ or $M' \cup -M'$. Then the lift of $\mathcal{N}$ is strongly monotone with respect to $K'$.
\end{thm}

\begin{proof}
By Lemma~\ref{lem:extrem points characterize}, the set $M'$ consists of vectors in $M$. We verify the conditions of Lemma~\ref{lem:subtract_off_strong_2} for the lifted cone $K'$. Throughout, we slightly abuse notation by letting $v$ denote both a vector in $M'$ and its lift, where the meaning is clear from context.

Let $F \subset K'$ be a proper face. Let $M' = \{v_1, \dots, v_n\}$ and define $m = \{v_1, \dots, v_k\} \subseteq M$ as the subset of vectors from $M$ whose lifts lie in $F$. Define a set of reactions $L = \{\Gamma_{i_1}, \dots, \Gamma_{i_h}\} \subseteq \mathcal{N}$ such that for each $\Gamma_{i_j} \in L$, there exist $v_1, v_2 \in m$ and $k \in \{0,1\}$ such that
\[
\Gamma_{i_j} = (-1)^k (v_1 - v_2).
\]
Since $F$ is a proper face, $L$ cannot contain all reactions in $\mathcal{N}$ (this is where the assumption on the dimensionality of $\operatorname{AffHul}(M)$ is used).

Let \( f \in F \), and consider the translated face \(F' = F - f \). Then \( \operatorname{AffHul}(F') \) is a linear subspace. Now consider any $\Gamma_i \in \mathcal{N}$. Suppose $\Gamma_i \in \operatorname{AffHul}(F')$ and all $v' \in m$ satisfy $v' \in Q_2(\Gamma_i)$. Then $\operatorname{AffHul}(F') \subseteq Q_2(\Gamma_i)$, which contradicts the assumption that $\Gamma_i \in \operatorname{AffHul}(F')$, since $Q_2(\Gamma_i)$ is orthogonal to directions in which $\Gamma_i$ has support. Therefore, there must exist $v' \in m$ such that $v' \notin Q_2(\Gamma_i)$.

By the strong connectivity of the $\mathcal{RI}$-graph, there exists a reaction $\Gamma_j \notin L$ that is adjacent to some $\Gamma_i \in L$, or (if $L = \emptyset$) a $\Gamma_j$ that shares support with some $v \in m$ (such a reaction must exist since every species participates in at least one reaction). If no such $\Gamma_j$ existed, then the reactions in $L$ would be disconnected from the rest of the network—i.e., there would be no directed path from them to all other reactions—contradicting strong connectivity.

Because $\Gamma_i \in \operatorname{AffHul}(F')$ and $\Gamma_j$ shares species with $\Gamma_i$ (or directly with some $v \in m$), we can find $k \in m$ such that $k$ and $\Gamma_j$ share support (i.e., $k \not\in Q_2(\Gamma_j))$. Since $M$ is closed under permissible operations, at least one of $k$, $k - \Gamma_j$, or $k + \Gamma_j$ belongs to $M$ and remains permissible. Let $k_i$ denote the resulting vector:
\[
k_i = k \pm \Gamma_j,
\]
choosing the sign so that the operation is permissible. By construction, $k_i \in M$, but its lift does not lie in $F$, since $\Gamma_j \notin \operatorname{AffHul}(F')$. Thus, the conditions of Lemma~\ref{lem:subtract_off_strong_2} are satisfied.

This establishes strong monotonicity with respect to $K'$ in the case where $0 \notin M'$.

Now consider the case where $0 \in M'$. Define a new set $\widetilde{M}' = M' \cup -M'$, which remains viable because $M$ is closed under permissible operations, and these operations preserve viability under negation. The zero vector is now a convex combination of vectors in $\widetilde{M}'$, and the dimension of $\operatorname{AffHul}(\widetilde{M}')$ equals that of $M$, since each element of $-M'$ lies in the affine span of $M'$, and the presence of the zero vector implies that the affine span is a linear subspace.

Therefore, by applying the same argument to $\widetilde{M}'$ in place of $M'$, we conclude that the lift of $\mathcal{N}$ is strongly monotone with respect to the lifted cone generated by $\widetilde{M}' = M' \cup -M'$.
\end{proof}

We now introduce the concept of \textit{aligned} matrices, which are convenient to work with in several of the results that follow.

\begin{enumerate}
    \item There exists a diagonal matrix \( D \) with positive diagonal entries such that each entry of \( \Gamma D \) is in \( \{0, 1, -1\} \).
    
    \item Any two reaction vectors share at most one coordinate where they have the same nonzero sign, and at most one coordinate where they have opposite nonzero signs.
    
    \item For each reaction vector, there is at most one nonzero coordinate that is zero in all other reaction vectors.
\end{enumerate}

\begin{define}
A matrix satisfying the above conditions is called an \textit{aligned} matrix.
\end{define}

\begin{thm}
\label{thm:P*gamma network strongly monotone}
Suppose \( \Gamma \) is an aligned matrix and that we are given a viable set of vectors. Let \( N \) be a matrix whose columns are these vectors. Suppose further that:

\begin{enumerate}
    \item \( P \) is a matrix with no zero columns, and each row of \( P \) contains exactly one nonzero entry.
    \item The matrix \( \Gamma \) is strongly monotone (respectively, monotone) with respect to the cone generated by the lift of the columns of \( N \).
\end{enumerate}

Then the network with stoichiometric matrix \( P\Gamma \) is strongly monotone (respectively, monotone) with respect to the lift of the column vectors of \( PN \).
\end{thm}

\begin{proof}
Assume first that \( \Gamma \) is strongly monotone with respect to the lift. Then the conditions of Lemma~\ref{lem:subtract_off_strong_2} are satisfied. Let \( K \) denote the cone generated by the lift of \( N \), and let \( PK \) denote the lift of \( PN \).

There is a bijection between proper faces \( F \subset K \) and faces \( PF \subset PK \), defined by \( PF = \{x \mid \exists y \in F \text{ such that } x = Py \} \). Because \( P \) has exactly one nonzero entry per row and no zero columns, it is injective and surjective onto its image, so this map between faces is well-defined and bijective.

Fix a face \( PF \subset PK \). By Lemma~\ref{lem:subtract_off_strong_2}, for the corresponding face \( F \subset K \), there exist \( \Gamma_i \), \( k \in F \), and \( k_i \in K \setminus F \) such that:
\begin{enumerate}
    \item \( k \notin Q_2(\Gamma_i) \),
    \item \( \Gamma_i = (-1)^j(k - k_i) \) for some \( j \in \{0,1\} \).
\end{enumerate}

We claim that the triple \( P\Gamma_i, Pk, Pk_i \) satisfies the same conditions for the face \( PF \subset PK \). This follows from the fact that \( P \) preserves the qualitative regions: if \( v \in Q_1^+(\Gamma_i) \), then \( Pv \in Q_1^+(P\Gamma_i) \); similarly for \( Q_1^- \) and \( Q_2 \). Therefore:
\begin{enumerate}
    \item \( Pk \notin Q_2(P\Gamma_i) \),
    \item \( P\Gamma_i = (-1)^j(Pk - Pk_i) \),
    \item \( Pk_i \notin PF \).
\end{enumerate}

Hence, the conditions of Lemma~\ref{lem:subtract_off_strong_2} are satisfied for \( P\Gamma \), and the network is strongly monotone with respect to the cone generated by the columns of \( PN \).

For the case of monotonicity, a similar argument applies using Lemma~\ref{lem:subtract_off}. Suppose \( \Gamma \) is monotone with respect to \( K \). Given a reaction vector \( \Gamma_i \), and vectors \( k \in Q_1^+(\Gamma_i) \), \( k' \in Q_1(-\Gamma_i) \cap K \), such that \( \Gamma_i = k - k' \), we apply \( P \) to obtain:
\[
P\Gamma_i = Pk - Pk', \quad Pk \in Q_1^+(P\Gamma_i), \quad Pk' \in Q_1(-P\Gamma_i) \cap PK.
\]
Thus, Conditions 2 and 3 of Lemma~\ref{lem:subtract_off} are preserved. Condition 1 also holds since, as before, \( P \) preserves region membership: if \( v \in Q_1(\Gamma_i) \cup Q_1(-\Gamma_i) \), then \( Pv \in Q_1(P\Gamma_i) \cup Q_1(-P\Gamma_i) \).

This completes the proof.
\end{proof}

\section{Cross-Polytope Cones}
\label{sec:cross-polytopes cones}

We are now in a position to prove new results about the dynamics of a special class of networks. Specifically, we show that this class is monotone with respect to cones whose associated poset structure corresponds to either a cross-polytope or a simplex.

\begin{define}
    Let \( e_i = (0, 0, \dots, 0, 1, 0, \dots, 0) \) denote the \( i \)-th standard basis vector in \( \mathbb{R}^n \), with a \( 1 \) in the \( i \)-th position and zeros elsewhere.

    A \textit{standard simplex} is a set of the form
    \[
    S = \left( \bigcup_i \{ e_{j_i} \} \right) \cup \left( \bigcup_l \{ -e_{k_l} \} \right),
    \]
    such that for no index \( i \) does the set \( S \) contain both \( e_i \) and \( -e_i \); that is, \( \{ e_i, -e_i \} \not\subset S \) for any \( i \). A \textit{simplicial cone} is defined as the conical hull (or lift) of a standard simplex.

    A \textit{standard cross-polytope} is a set of the same form,
    \[
    S = \left( \bigcup_i \{ e_{j_i} \} \right) \cup \left( \bigcup_l \{ -e_{k_l} \} \right),
    \]
    with the additional condition that for each index \( i \), either both \( e_i \) and \( -e_i \) are in \( S \), or neither is; that is,
    \[
    \text{for all } i, \quad \{ e_i, -e_i \} \subset S \quad \text{or} \quad \{ e_i, -e_i \} \cap S = \emptyset.
    \]
    A \textit{cross-polytope cone} is the lift of a standard cross-polytope.
\end{define}

\begin{define}
    An aligned matrix is said to be of \textit{type C} if it satisfies the following conditions:
    \begin{enumerate}
        \item Every entry belongs to the set \(\{ -2, -1, 0, 1, 2 \}\).
        \item Each column has \( \ell_1 \)-norm equal to 2.
        \item The matrix has no zero rows.
    \end{enumerate}
\end{define}

\begin{thm}
\label{thm:type C strongly monotone}
If the stoichiometric matrix \( \Gamma \) with a strongly connected $\mathcal{RI}$-graph is of type C, then the corresponding reaction network is strongly monotone with respect to a cone whose poset structure is that of either a cross-polytope or a simplex.
\end{thm}

\begin{proof}
Let \( \Gamma \) be a stoichiometric matrix of type C. Each column of \( \Gamma \) has \( \ell_1 \)-norm equal to 2 and entries in \( \{-2, -1, 0, 1, 2\} \). Therefore, each column is one of the following:
\begin{enumerate}
    \item A column with two nonzero entries, both equal to \( 1 \),
    \item A column with one entry equal to \( 1 \) and one equal to \( -1 \),
    \item A column with a single nonzero entry equal to \( \pm 2 \).
\end{enumerate}

Let \( K \subseteq \mathbb{R}^n \) denote the set of all signed standard basis vectors:
\[
K = \left\{ e_i, -e_i \mid 1 \leq i \leq n \right\}.
\]
We claim that \( K \) is a viable set in the sense of Definition~\ref{def:viable set}.

To verify that \( K \) is closed (Definition~\ref{def:closed}), consider any \( v \in K \) and any column \( \Gamma_j \) of \( \Gamma \). If \( v \notin Q_2(\Gamma_j) \), then by the structure of \( \Gamma \), either \( v - \Gamma_j \in K \) or \( v + \Gamma_j \in K \). For example:
\begin{enumerate}
    \item If \( \Gamma_j = e_i + e_k \) and \( v = e_i \), then \( v - \Gamma_j = -e_k \in K \).
    \item If \( \Gamma_j = e_i - e_k \) and \( v = e_i \), then \( v - \Gamma_j = e_k \in K \).
    \item If \( \Gamma_j = 2e_i \) and \( v = e_i \), then \( v - \Gamma_j = -e_i \in K \).
\end{enumerate}
In each case, the result remains in \( K \), showing that \( K \) is closed under permissible operations.

To verify the permissibility condition for viability, observe that each \( k \in K \) has only one nonzero entry. Thus, for any reaction vector \( \Gamma_j \), either \( k \in Q_1(\Gamma_j) \) or \( k \in Q_1(-\Gamma_j) \). Hence, all elements of \( K \) satisfy the viability condition, and we conclude that \( K \) is viable.

Next, since \( \Gamma \) has no zero rows, every species appears in at least one reaction. Therefore, for each index \( i \), at least one of \( e_i \) or \( -e_i \) is involved in a permissible operation with some \( \Gamma_j \).

Let \( S \subseteq K \) be the minimal viable set generated by starting from any \( v \in K \) that shares support with a column of \( \Gamma \), and iteratively applying permissible operations. Because the $\mathcal{RI}$-graph (whose nodes are reaction vectors connected by shared species) is strongly connected, this construction propagates through all indices.

Moreover, all new points in \( S \) are generated by adding or subtracting reaction vectors, so \( \operatorname{AffHul}(S) \) has dimension at most that of the stoichiometric compatibility class. By Lemma~\ref{lem_viable_at_least_stoich_dimension}, this implies equality of dimension. Therefore, \( S \) is a viable set satisfying the hypotheses of Theorem~\ref{thm:output_M_to_strong_monotone}. It follows that the system is strongly monotone with respect to the lifted cone generated by \( S \).

We now characterize the structure of this cone:
\begin{enumerate}
    \item If there exists an index \( i \) such that both \( e_i \in S \) and \( -e_i \in S \), then by the strong connectedness of the $\mathcal{RI}$-graph and closure of \( S \), it follows that all elements of \( K \) are eventually included in \( S \). Indeed, if we can generate a point \( e_i - \Gamma_j \in S \), then the vector \( \Gamma_j - e_i = - (e_i - \Gamma_j) \) is also in \( S \). In this case, \( S = K \), and the corresponding cone is a cross-polytope cone.
    
    \item If no such index exists, then for each \( i \), at most one of \( e_i \) or \( -e_i \) lies in \( S \). In this case, \( S \) forms a standard simplex, and the corresponding cone is simplicial.
\end{enumerate}

Therefore, the reaction network is strongly monotone with respect to a cone whose poset structure is either that of a cross-polytope or a simplex.
\end{proof}

We in fact have a bit more information about when a type C reaction network is monotone with respect to a simplicial cone:

\begin{lem}
\label{lem:type C simplicial iff dependent rows}
A type C network is monotone with respect to a simplicial cone if and only if it has linearly dependent rows.
\end{lem}

\begin{proof}
First, suppose that \( \Gamma \) has linearly dependent rows. Then there exists a nonzero vector \( v \) such that \( v^T \Gamma = 0 \). We claim that \( \Gamma \) cannot have any column with only a single nonzero entry. Indeed, if some column \( \Gamma_j \) has only one nonzero entry in coordinate \( i \), then \( (v)_i = 0 \). But if another column has nonzero entries in coordinates \( i \) and \( j \), then the equation \( v^T \Gamma = 0 \) forces \( (v)_j = 0 \) as well. Repeating this argument would eventually imply \( v = 0 \), contradicting our assumption. 

Therefore, every column of \( \Gamma \) must have at least two nonzero entries. Since \( \Gamma \) is type C, all entries lie in \( \{-1, 0, 1\} \), and each column must have exactly two nonzero entries.

Now consider such a vector \( v \) satisfying \( v^T \Gamma = 0 \). We can scale \( v \) so that each nonzero entry lies in \( \{ -1, 1 \} \). Define a simplicial cone generated by the vectors \( e_i \) for which \( (v)_i > 0 \), and \( -e_i \) for which \( (v)_i < 0 \). Without loss of generality, assume \( (v)_i > 0 \), so \( e_i \) lies in the cone.

Let \( \Gamma_j \) be a column of \( \Gamma \) with nonzero entries in coordinates \( i \) and \( k \). Then:
\[
(v)_i (\Gamma_j)_i + (v)_k (\Gamma_j)_k = 0 \quad \Rightarrow \quad (\Gamma_j)_k = -\frac{(v)_i}{(v)_k} (\Gamma_j)_i.
\]
Assume \( (\Gamma_j)_i > 0 \). Then the difference \( e_i - \Gamma_j \) will have a nonzero component in the \( k \)-th coordinate of the form \( \mathrm{sgn}((\Gamma_j)_k) e_k \). From the equation above, we have \( \mathrm{sgn}((\Gamma_j)_k) = \mathrm{sgn}((v)_k) \), so \( e_i - \Gamma_j \) lies in the cone. The same reasoning applies under sign changes of \( (v)_i \) or \( (\Gamma_j)_i \). Hence, the cone is forward-invariant under the reactions, and the system is monotone with respect to it.

Conversely, suppose \( \Gamma \) is a type C network that is monotone with respect to the lift of a simplicial set \( S = \{v_1, v_2, \dots, v_n\} \), where each \( v_i = \pm e_i \). Since the set is simplicial and composed of signed standard basis vectors, any column of \( \Gamma \) with only one nonzero entry would violate forward-invariance of the cone. Hence, every column of \( \Gamma \) must have exactly two nonzero entries.

Now define a vector \( v \in \mathbb{R}^n \) by setting \( (v)_1 = 1 \), and for each \( i \), define
\[
(v)_i = \mathrm{sgn}(v_i) \cdot \mathrm{sgn}(v_1),
\]
where \( \mathrm{sgn}(v_i) \) refers to whether \( v_i = e_i \) or \( v_i = -e_i \).

Now consider any column \( \Gamma_j \) with nonzero entries in coordinates \( j \) and \( k \). From monotonicity, we must have \( v_j \pm \Gamma_j \in S \), depending on the direction of the vector field. This implies:
\[
(\Gamma_j)_j (\Gamma_j)_k = -\mathrm{sgn}(v_j) \cdot \mathrm{sgn}(v_k).
\]
We now compute:
\[
v \cdot \Gamma_j = (v)_j (\Gamma_j)_j + (v)_k (\Gamma_j)_k.
\]
Substitute the expressions for \( (v)_j \) and \( (v)_k \) using the definition of \( v \):
\[
v \cdot \Gamma_j = \mathrm{sgn}(v_j)\mathrm{sgn}(v_1)(\Gamma_j)_j + \mathrm{sgn}(v_k)\mathrm{sgn}(v_1)(\Gamma_j)_k = \mathrm{sgn}(v_1)\left[ \mathrm{sgn}(v_j)(\Gamma_j)_j + \mathrm{sgn}(v_k)(\Gamma_j)_k \right].
\]
But by the identity above, we know \( \mathrm{sgn}(v_j)(\Gamma_j)_j = -\mathrm{sgn}(v_k)(\Gamma_j)_k \), so the sum is zero. Hence \( v \cdot \Gamma_j = 0 \) for every column \( \Gamma_j \), implying \( v^T \Gamma = 0 \), and thus the rows of \( \Gamma \) are linearly dependent.
\end{proof}

 We also have the following:

\begin{cor}
\label{cor:type C is weakly contractive}
    Type C reaction networks are weakly contractive with respect to some norm. 
\end{cor}

\begin{proof}
    By Theorem \ref{thm:strongly_monotone_lift_to_weakly contractive_general_2}, the cone our system is strongly monotone with respect to (from Theorem \ref{thm:type C strongly monotone}) implies we can construct a norm for which we are weakly contractive. 
\end{proof}

\section{Additional Weakly Contractive Systems}
\label{sec:additional weakly contractive systems}
Using our methods we can show that several constructions (as well as some new constructions) in the literature are weakly contractive. We will first need an additional way to associate aligned matrices with more general stoichiometric matrices.

\begin{define}
    We say a stoichiometric matrix $\Gamma$ is \textit{alignable} if we can write it as $PND$ where 

    \begin{enumerate}
        \item $N$ is an aligned matrix with $-1,0,$ or $1$ entries.
        \item $P$ has no 0 columns or 0 rows and each row has exactly one nonzero entry.
        \item $D$ is a diagonal matrix with nonzero entries on its diagonal.
    \end{enumerate}

     We call $N$ in this factorization the \textit{alignment}.
\end{define}

\begin{define}
\label{def:equivalence for rows}
    Suppose we have a matrix $\Gamma$. Define an equivalence relation on the rows as follows: we say two rows $i$ and $j$ are related, written $i \sim j$, iff 

    \begin{enumerate}
        \item For all columns $\Gamma_k \in \Gamma$, $(\Gamma_k)_i \neq 0$ iff $(\Gamma_k)_j \neq 0$.
        \item For all columns $\Gamma_{k_1}, \Gamma_{k_2} \in \Gamma$, we have that $\mbox{sign} ((\Gamma_{k_1})_i(\Gamma_{k_1})_j) = \mbox{sign} ((\Gamma_{k_2})_i(\Gamma_{k_2})_j)$.
    \end{enumerate}

    We define an equivalence relation similarly for a set of vectors, viewing them as columns of a matrix. We call this equivalence relation for sets of vectors or matrices the \textit{coordinate equivalence}.
\end{define}

The main idea behind these definitions is that they allow us to simplify certain reaction networks. If we can write \( \Gamma = PND \), then for our purposes, it suffices to consider the simplified network with stoichiometric matrix \( \Gamma' = N \). In particular, we have the following:

\begin{cor}
\label{cor:align to general}
Suppose the stoichiometric matrix \( \Gamma = PND \) is alignable. If the simplified stoichiometric matrix \( \Gamma' = N \) is strongly monotone (respectively, monotone) with respect to a cone \( K \), then \( \Gamma \) is strongly monotone (respectively, monotone) with respect to the cone \( PK \).
\end{cor}

\begin{proof}
Multiplying on the right by \( D^{-1} \) yields an equivalent reaction network, so \( \Gamma = PND \) is strongly monotone (respectively, monotone) if and only if \( PN \) is. By Theorem~\ref{thm:P*gamma network strongly monotone}, if \( N \) is strongly monotone (respectively, monotone) with respect to \( K \), then \( PN \) is strongly monotone (respectively, monotone) with respect to \( PK \), as desired.
\end{proof}

\begin{lem}
\label{lem:cube type to aligned}
Let \( \Gamma \) be a matrix with at most two nonzero entries per row. Suppose that whenever two reactions \( R_i \) and \( R_j \) share a set of species \( I \), the corresponding subvectors \( (R_i)_I \) and \( (R_j)_I \) satisfy for each $l \in I$ that either \( (R_i)_l = (R_j)_l \) or \( (R_i)_l = - (R_j)_l \). Then \( \Gamma \) is alignable.
\end{lem}

\begin{proof}
Let \( \Gamma \) have \( n \) rows. Define coordinate equivalence classes \( [I]_{\nu} \) for \( \nu \in \{1,2,\dots,n\} \). For each class \( [I]_{\nu} \), select a representative reaction vector \( \Gamma_{\nu_i} \) such that \( \Gamma_{\nu_i} \) is nonzero for every species in \( [I]_{\nu} \).

Construct a matrix \( P \) where column \( i \) has a nonzero entry in row \( j \) if \( j \in [I]_i \), and set \( P_{ij} = (\Gamma_{\nu_i})_j \) for the representative \( \Gamma_{\nu_i} \); elsewhere, \( P \) is zero. Each column of \( P \) corresponds to an equivalence class, so each row has at most one nonzero entry, and there is no column of all zeros. Define a matrix \( N \) with entries:
\[
N_{ij} =
\begin{cases}
    1, & \text{if reaction vector } \Gamma_j \text{ has entries in } [I]_\nu \text{ matching } \Gamma_{\nu_i}, \\
   -1, & \text{if the entries are opposite}, \\
    0, & \text{if all corresponding entries are zero}.
\end{cases}
\]
Then, \( \Gamma = PN \), where \( N \) is aligned.

Indeed, the matrix \( N \) satisfies:
\begin{enumerate}
    \item Entries are in \( \{1,-1,0\} \), encoding whether \( \Gamma_i \) aligns positively, negatively, or does not overlap with the corresponding column in \( P \).
    \item Any two columns share at most two species with opposite signs. If two columns \( v_1 \) and \( v_2 \) shared three rows \( k_1, k_2, k_3 \), there would exist at least two rows, say \( k_1, k_2 \), such that \( \operatorname{sgn} (N_{k_1,v_1} N_{k_1,v_2}) = \operatorname{sgn} (N_{k_2,v_1} N_{k_2,v_2}) \). This would imply that the species in the equivalence classes of columns \( k_1 \) and \( k_2 \) in \( P \) should be merged into the same equivalence class.
    \item At most one nonzero entry per row where all other columns have zero. If two such entries existed in rows \( k_1 \) and \( k_2 \), then the equivalence classes in columns \( k_1 \) and \( k_2 \) of \( P \) must belong to a single larger equivalence class.
\end{enumerate}
Thus, \( N \) is aligned, completing the proof.
\end{proof}

\subsection{Cubical Cones}
\label{sec:cubical_cones}

Before examining the monotone systems in \cite{BANAJI20131359}, we will describe a general cubical cone construction.

\begin{define}
    We say an aligned network is \textit{cubical} if it satisfies the following properties:
    \begin{enumerate}
        \item Each row contains at most two nonzero entries.
        \item The columns are linearly independent.
    \end{enumerate}
\end{define}

\begin{thm}
\label{thm:cubical network cones}
    Cubical networks are monotone with respect to a cubical cone (i.e., a cone with the poset of a cube). 
\end{thm}

\begin{proof}

Define an initial vector \( v \) with coordinates:
    \[
    (v)_i =
    \begin{cases}
        1, & \text{if all reaction vectors have a nonnegative \( i \)th coordinate,} \\
        -1, & \text{if all reaction vectors have a nonpositive \( i \)th coordinate,} \\
        0, & \text{otherwise.}
    \end{cases}
    \]
    Note that if \( v = 0 \), then each row of the stoichiometric matrix must either consist entirely of zeros or contain exactly one \(-1\) and one \(1\). Summing all reaction vectors would then yield the zero vector, contradicting their linear independence. Hence, this procedure ensures that \( v \neq 0 \). By construction, \( v \in Q_2(\Gamma_i) \cup Q_1^+(\Gamma_i) \) for all \( i \).

    Define the set:
    \[
    A = \{ v - \sum_{i \in J}  \Gamma_i \mid J \subset \{1,2,\dots,n\}
     \}.
    \]
    We must show that:
    1. \( A \) forms a cube.
    2. \( A \) is viable. Once we do this, by Theorem \ref{thm:lift viable for monotone} our system is indeed monotone with respect to the lift of this set, which is a cubical cone.

 To confirm the cube structure, consider an invertible transformation \( T \) such that \( T(R_i) = e_i \), where \( e_i \) is the \( i \)th standard basis vector. This maps our set to:
\[
 T(v) - \sum_{i \in J} e_i,
 \]
 which is a cube.

 Note that by construction, \( v \in Q_1(\Gamma_i) \) for all \( i \). Each vector in \( A \) has entries in \( \{0,1,-1\} \). Let \( v_J \in A \) be the vector corresponding to the subset \( J \) in the definition of \( A \). To show that \( A \) is closed, it suffices to verify that  
\[
i \in J \implies v_J \in Q_1(-\Gamma_i), \quad i \notin J \implies v_J \in Q_1(\Gamma_i).
\]

We proceed by case analysis.

\textit{Case 1:} \( i \in J \).  
Each coordinate of \( v_J \) that overlaps with \( \Gamma_i \) must be nonpositive. Consider an arbitrary such coordinate \( l \), and assume without loss of generality that \( (\Gamma_i)_l = -1 \).

If \( \Gamma_i \) is the only reaction involving coordinate \( l \), then \( (v_J)_l = 0 \).

Now suppose there is another reaction \( \Gamma_j \) such that \( (\Gamma_j)_l = -1 \). Since \( (v)_l = -1 \), it follows that:
\[
(v_J)_l =
\begin{cases}
0, & \text{if } j \notin J, \\
1, & \text{if } j \in J.
\end{cases}
\]
Next, suppose instead that \( (\Gamma_j)_l = 1 \). Since \( (v)_l = 0 \), we have:
\[
(v_J)_l =
\begin{cases}
0, & \text{if } j \in J, \\
1, & \text{if } j \notin J.
\end{cases}
\]
In all cases, these conditions ensure that \( v_J \in Q_1(-\Gamma_i) \).

\textit{Case 2:} \( i \notin J \).  
Each coordinate of \( v_J \) that overlaps with \( \Gamma_i \) must be nonnegative. Again, consider an arbitrary such coordinate \( l \) and assume \( (\Gamma_i)_l = -1 \).

If \( \Gamma_i \) is the only reaction involving coordinate \( l \), then \( (v_J)_l = -1 \).

Now suppose there is another reaction \( \Gamma_j \) such that \( (\Gamma_j)_l = -1 \). Since \( (v)_l = -1 \), it follows that:
\[
(v_J)_l =
\begin{cases}
-1, & \text{if } j \notin J, \\
0, & \text{if } j \in J.
\end{cases}
\]
Next, suppose instead that \( (\Gamma_j)_l = 1 \). Since \( (v)_l = 0 \), we have:
\[
(v_J)_l =
\begin{cases}
-1, & \text{if } j \in J, \\
0, & \text{if } j \notin J.
\end{cases}
\]
In all cases, these conditions ensure that \( v_J \in Q_1(\Gamma_i) \).

Since the set is closed under permissible operations, and for each reaction \( \Gamma_i \) there exists some \( k \in A \) such that \( k \in Q_1(\Gamma_i) \cup Q_1(-\Gamma_i) \), we conclude that \( A \) is a viable set.
 \end{proof}

We will call the cones from \cite{BANAJI20131359} \textit{type S} cones, and the corresponding networks \textit{type S} networks. The cones are of the form $\Lambda = c 1^t + \Gamma B$ where $c$ is some column vector, $1^t$ is the transpose of a vector of all 1's, $\Gamma$ is some stoichiometric matrix, and $B$ contains as columns the vertices of a cube, e.g., $B = \begin{bmatrix} 0 & 1 & 0 & 1 \\ 0 & 0 & 1 & 1 \end{bmatrix}$. 

We will also require a new type of graph introduced in \cite{BANAJI20131359}. Given two matrices \( A \in \mathbb{R}^{m \times n} \) and \( B \in \mathbb{R}^{n \times m} \), we define the directed bipartite graph \( G_{A,B} \) as follows. The graph has \( m \) vertices labeled \( u_1, \ldots, u_m \) and \( n \) vertices labeled \( v_1, \ldots, v_n \). For each pair \( (i, j) \), we include a directed edge \( u_i \to v_j \) iff \( A_{ij} \neq 0 \), and a directed edge \( v_j \to u_i \) iff \( B_{ji} \neq 0 \).

A type $S$ network $\Gamma$ satisfies the following: 

\begin{enumerate}
    \item $\Gamma$ has linearly independent columns.
    \item There exists a constant vector $c$ such that $\Lambda = c 1^t + \Gamma B$ satisfies:
    \begin{enumerate}
        \item There exists a nonnegative matrix $P$ such that $\Lambda P = I$.
        \item There exists a diagonal matrix $D$ with positive entries such that $D \Lambda$ has entries only consisting of $-1,1$ or $0$.
    \end{enumerate}
    \item The graph $G_{\Gamma, \partial R}$ is strongly connected
\end{enumerate}

\begin{lem}
\label{lem:DSR strong to RI strong}
The \( \mathcal{RI} \)-graph of a network is strongly connected if and only if the \( G_{\Gamma, \partial R} \) graph is strongly connected.
\end{lem}

\begin{proof}
First, assume that \( G_{\Gamma, \partial R} \) is strongly connected. Interpret the \( u_i \) nodes of \( G_{\Gamma, \partial R} \) as species (i.e., rows of \( \Gamma \)) and the \( v_i \) nodes as reactions (i.e., columns of \( \Gamma \)). Since the graph is strongly connected, there exists a directed path between any two reactions \( v_1 \) and \( v_2 \). Consider a portion of such a path involving nodes \( v_1 \to u \to v_i \). The edge \( u \to v_i \) implies that species \( u \) is affected by reaction \( v_i \), and the edge \( v_1 \to u \) implies that \( u \) influences the kinetics of reaction \( v_1 \). Therefore, in the \( \mathcal{RI} \)-graph, there must be a directed edge \( v_i \to v_1 \). Repeating this argument along the path from \( v_1 \) to \( v_2 \) yields a reversed path in the \( \mathcal{RI} \)-graph from \( v_2 \) to \( v_1 \). Since such reversed paths exist for all reaction pairs, we conclude that the \( \mathcal{RI} \)-graph is strongly connected.

Now assume that the \( \mathcal{RI} \)-graph is strongly connected. The argument is similar. If there is a directed edge from reaction \( v_1 \) to \( v_2 \), then there exists a species \( u \) such that \( v_1 \) affects \( u \), and \( u \) influences the kinetics of \( v_2 \). In \( G_{\Gamma, \partial R} \), this corresponds to the edges \( v_2 \to u \) and \( u \to v_1 \). Applying this construction to all edges in the \( \mathcal{RI} \)-graph shows that for any pair of reactions \( v_i \) and \( v_j \), there is a directed path in \( G_{\Gamma, \partial R} \) connecting them. Additionally, each species \( u_i \) has both incoming and outgoing edges in \( G_{\Gamma, \partial R} \): it appears in at least one reaction and affects the kinetics of at least one reaction. (This follows from our standing assumption that every species influences the kinetics of some reaction—otherwise, it would not affect the dynamics.) Hence, \( G_{\Gamma, \partial R} \) is strongly connected.
\end{proof}

\begin{thm}
\label{thm:type S align}
    Type S networks are alignable, with a cubical alignment.
\end{thm}
    
\begin{proof}
    Define $\Gamma' = D\Gamma$, $c' = Dc$, and $\Lambda' = D\Lambda = c'1^t + \Gamma'B$. Since $\Lambda'$ has entries only in $\{-1, 0, 1\}$, it follows that $c'$ must also consist of entries in $\{-1, 0, 1\}$, as $c'$ is simply the first column of $\Lambda'$.

    Next, define $P' = P D^{-1}$. Then
    \[
        \Lambda' P' = D\Lambda P D^{-1} = I,
    \]
    so $\Lambda'$ is right-invertible via a nonnegative matrix.

    Now observe that each row of $\Gamma'$ contains at most three distinct values, since $\Gamma'$ appears as the first few columns of $\Gamma'B = \Lambda' - c'1^t$, and $\Lambda'$ has entries only in $\{-1, 0, 1\}$. Hence, each row of $\Gamma'B$ also has at most three distinct values.

    Suppose, for contradiction, that a row of $\Gamma'$ has nonzero entries in three distinct columns, say columns $i$, $j$, and $k$. Then the columns of $\Gamma'B$ corresponding to $i$, $j$, $k$, $i+j$, $i+k$, $j+k$, and $i+j+k$ (i.e., linear combinations over $\{0,1\}$ of those columns) will yield at least four distinct values in that row: namely $0$, the individual entries, and their sums. This contradicts the assumption that each row of $\Gamma'B$ has at most three distinct values. Therefore, each row of $\Gamma'$ has at most two nonzero entries.

    Since $\Gamma' = D\Gamma$, the same holds for $\Gamma$: each row has at most two nonzero entries.

    Next, note that all entries in $\Gamma'$ lie in the set $\{-2, -1, 0, 1, 2\}$. This follows from the fact that $\Gamma'B = \Lambda' - c'1^t$ and both $\Lambda'$ and $c'$ have entries in $\{-1, 0, 1\}$. Furthermore, any two nonzero entries in a given row must differ by at most 2, as this holds for entries of $\Lambda'$ and hence for $\Gamma'B$.

    Suppose now that a row of $\Gamma'$ contains an entry $a = \pm 2$ in some column $i$. Then $a$ must be the only nonzero entry in that row. Otherwise, if another column $j$ has a nonzero entry $b$, then among the values $\{0, a, b, a + b\}$, there must be two differing by more than 2 in magnitude, again contradicting the constraint on entry differences in $\Gamma'B$.

    Define a diagonal matrix $D_a$ as follows: set the $i$-th diagonal entry of $D_a$ to $1/2$ if the $i$-th column of $\Gamma'$ contains an entry of magnitude 2 in any row; otherwise, set it to $1$. Then $D_a \Gamma'$ has entries only in $\{-1, 0, 1\}$ and each row has at most two nonzero entries.

    By Lemma~\ref{lem:cube type to aligned}, this implies that $D_a \Gamma' = P_a N$ for some matrix $P_a$ with exactly one nonzero entry per row and column and some aligned matrix $N$ with linearly independent columns and at most two nonzero entries per row.

    Combining all steps, we obtain the factorization
    \[
        \Gamma = D^{-1} D_a^{-1} P_a N,
    \]
    showing that $\Gamma$ is alignable with cubical alignment.
\end{proof}

\begin{thm}
\label{thm:cubicals are weakly contractive}
    Cubical networks, and in particular Type S networks, are weakly contractive.
\end{thm}

\begin{proof}
 First by Theorem \ref{thm:cubical network cones} and Lemma \ref{thm:lift viable for monotone} we know our system is monotone with respect to a cone. By Theorem \ref{thm:output_M_to_strong_monotone} our system is strongly monotone with respect to the lifted cone. By Lemma \ref{lem:DSR strong to RI strong} we know they have a strongly connected \( \mathcal{RI} \)-graph. Now we can use Theorem \ref{thm:strongly_monotone_lift_to_weakly contractive_general_2} to conclude the systems are weakly contractive and we are done.
 
\end{proof}

The following is an example of a cubical network which is not a type S network (since type S matrices cannot have a row including entries with a ratio of 3:1, this follows from the proof of Theorem \ref{thm:type S align} since in the proof $\Gamma' = D \Gamma $ can only have entries of $\{-2,-1,0,1,2\}$, from which it follows row entries of $\Gamma$ cannot be 3:1), and thus the network is not one of those considered in \cite{BANAJI20131359}. As far as the authors are aware, this network is not contained in any previously known classes of globally convergent networks. Suppose we have the stoichiometric matrix
$$\Gamma = \begin{bmatrix}
    -3 & 0 & 0 & 1 \\
    1 & -1 & 0 & 0 \\
    1 & 0 & -1 & 0 \\
    0 & 1 & 1 & 0 \\
    0 & 0 & 0 & -1
\end{bmatrix}$$

We can factor our matrix as:
\[
\Gamma = PND = 
    \begin{bmatrix}
    3 & 0 & 0 & 0 & 0\\
    0 & 1 & 0 & 0 & 0 \\
    0 & 0 & 1 & 0 & 0 \\
    0 & 0 & 0 & 1 & 0 \\
    0 & 0 & 0 & 0 & 3
\end{bmatrix}
\begin{bmatrix}
    -1 & 0 & 0 & 1 \\
    1 & -1 & 0 & 0 \\
    1 & 0 & -1 & 0 \\
    0 & 1 & 1 & 0 \\
    0 & 0 & 0 & -1
\end{bmatrix}
\begin{bmatrix}
    1 & 0 & 0 & 0 \\
    0 & 1 & 0 & 0 \\
    0 & 0 & 1 & 0 \\
    0 & 0 & 0 & 1/3
\end{bmatrix}.
\]

\subsection{Connection to a Class of Simplicial Cones}
\label{sec:type_I_networks}

A class of networks monotone with respect to a class of simplicial cones is introduced in \cite{doi:10.1137/120898486}. We will refer to these networks as \textit{type I} networks. More precisely, we have the following definition (taken from \cite{doi:10.1137/120898486}):

\begin{define}
\label{def:type I networks}
    A \textit{type I} network is a network satisfying the following conditions:

    \begin{enumerate}
        \item We can write $\Gamma = \Lambda \Theta$ where
        \begin{enumerate}
            \item The matrix $\Lambda$ has no 0 columns, and each row has exactly one nonzero entry.
            \item The matrix $\Theta$ satisfies that each column has at most 2 nonzero entries, and nonzero entries in the same column have opposite signs. Also, $\mbox{ker} (\Theta^t)$ is one-dimensional and contains a positive vector.
        \end{enumerate}
        \item The graph $G_{\Gamma, \partial R}$ is strongly connected
    \end{enumerate}
\end{define}

\begin{lem}
    All type I networks are alignable, and their alignment is type C.
\end{lem}

\begin{proof}
    Note from Condition 1b that the matrix $\Theta$ in $\Gamma = \Lambda \Theta$ must contain at most two nonzero entries of different sign in each column. Next factor $\Theta = D_1 \Theta'$ where $D_1$ is a diagonal matrix with entries the same as $\ker(\Theta^t)$. Note by Condition 1b that each entry in a column of $\Theta'$ must now have the same absolute value, so we can write $\Theta' = \Theta'' D_2$ where $D_2$ is a diagonal matrix with positive entries, chosen so that $\Theta''$ contains only $-1,1$ or $0$. Note that due to the kernel having a positive vector, each column contains 2 nonzero entries or is a zero column. Thus $\Theta''$ is our desired alignment, satisfying the conditions to be type C.
\end{proof}

\begin{thm}
    All type I networks are weakly contractive.
\end{thm}

\begin{proof}
    By Lemma \ref{lem:DSR strong to RI strong} we know they have a strongly connected \( \mathcal{RI} \)-graph. Now weak contractivity follows immediately from Corollary \ref{cor:type C is weakly contractive}.
\end{proof}

We also have the following equivalence:

\begin{lem}
    Suppose we have an aligned matrix with at most two nonzero entries per column. Then it is type I if and only if it has linearly dependent rows. 
\end{lem}

\begin{proof}

    By the proof of Lemma \ref{lem:type C simplicial iff dependent rows} if a network has linearly independent rows then each column has exactly two nonzero entries. If we have a vector $v$ such that $v^t \Gamma = 0$, simply take $\Lambda$ to be the matrix with diagonal values equal to the signs of the values of entries in $v$, and 0's elsewhere. Then $(\Lambda)(\Lambda^{-1} \Gamma)$ is the desired factorization to show that $\Gamma$ is type $I$.

    If the network is type $I$, it is given in the conditions for a type $I$ matrix that is has dependent rows.
\end{proof}

\subsection{Reaction Coordinates}
\label{sec:react_coordinates}

We have the following result:

\begin{thm}[Theorem 7 from \cite{ANGELI2007598}]
\label{thm:postive vect of fluxes equilibrium}
Suppose we have a persistent reaction network consisting only of irreversible reactions. Then there exists a vector \( v \) with strictly positive entries such that \( \Gamma v = 0 \).
\end{thm}

Recall that our definition of persistence includes the compactness of stoichiometric compatibility classes. Using the above theorem, we can prove the following:

\begin{lem}
Suppose we have a persistent reaction network \( \Gamma \) consisting only of irreversible reactions. Then the \( \mathcal{RI} \)-graph of this network is strongly connected.
\end{lem}

\begin{proof}
Assume, for contradiction, that the \( \mathcal{RI} \)-graph is not strongly connected. Consider the directed acyclic graph formed by collapsing each strongly connected component of the \( \mathcal{RI} \)-graph into a single node. Select a component with only outgoing edges and no incoming edges. Let \( \Gamma' \) denote the subnetwork consisting of the reactions within this component.

By Theorem~\ref{thm:postive vect of fluxes equilibrium}, there exists a vector \( v \gg 0 \) such that \( \Gamma v = 0 \). Since the reactions in \( \Gamma' \) do not share reactants with any other reactions outside the component, it follows that \( \Gamma' v = 0 \) for all species that appear as reactants in \( \Gamma' \). However, because the component has outgoing edges, there must be species in \( \Gamma' \) that appear only as products. As a result, \( \Gamma' v \) is a vector with nonnegative entries, at least one of which is strictly positive.

This contradicts the assumption that stoichiometric compatibility classes are compact, since \( \Gamma' v \neq 0 \) would allow unbounded accumulation of those species. Therefore, the \( \mathcal{RI} \)-graph must be strongly connected.
\end{proof}

Finally, note that every reaction network is equivalent to one consisting only of irreversible reactions, and the \( \mathcal{RI} \)-graphs of the two are strongly connected if and only if each other is. Thus, in general, persistence implies that the \( \mathcal{RI} \)-graph is strongly connected.

We will need a graphical definition:

\begin{define}
    We say our R-graph satisfies the \textit{signed loop} property if every edge in the graph has a sign, and every loop in the graph has an even number of positive edges.
\end{define}

It is shown in \cite{Angeli2010-hj} that this is equivalent to being monotone with respect to an orthant cone in reaction coordinates. We will not need the definition of reaction coordinates to demonstrate that all such reversible networks are weakly contractive, though roughly speaking, reaction coordinates keep track of the ``extent of reaction" instead of the amount of each species. This gives rise to an alternative dynamical system. 

Our definition of the signed loop property differs from \cite{Angeli2010-hj} in that our usage of signs is switched and we do not allow multiple edges between our nodes (our edges without a sign correspond to nodes that would have had two edges between them, one positive and one negative, per the definition form \cite{Angeli2010-hj}).

We will prove all the reaction networks that are monotone in reaction coordinates with respect to an orthant cone are monotone in species coordinates as well. We break the proof into two cases, depending on whether the reaction vectors are linearly independent.

\begin{lem}
\label{lem:reaction_coords_independent}
Let a reaction network be monotone with respect to an orthant cone in reaction coordinates (i.e., it satisfies the signed loop property). Assume the network has independent reaction vectors. Then it is also monotone with respect to a cone whose extreme rays are precisely the reaction vectors of the network, each included with a chosen sign.
\end{lem}

\begin{proof}
We construct a cone \( K \) whose extreme rays are (positive or negative) scalar multiples of the reaction vectors, and for which the conditions of Corollary~\ref{lem:reac_vec_extreme ray} are satisfied.

Start by selecting an arbitrary reaction vector \( \Gamma_1 \) and include \( \Gamma_1 \) as an extreme ray of \( K \). Then, for each reaction vector \( \Gamma_j \) adjacent to \( \Gamma_1 \) in the R-graph:
\begin{enumerate}
    \item If the edge between \( \Gamma_1 \) and \( \Gamma_j \) is positive, include \( -\Gamma_j \) in \( K \);
    \item If the edge is negative, include \( \Gamma_j \) in \( K \).
\end{enumerate}

Proceed inductively: for each unassigned reaction vector \( \Gamma_k \) adjacent to a previously assigned reaction vector \( \Gamma_j \in K \), assign \( \pm \Gamma_k \) to \( K \) based on the sign of the edge between \( \Gamma_j \) and \( \Gamma_k \), according to the following rules:
\begin{enumerate}
    \item If \( +\Gamma_j \in K \) and the edge is positive, assign \( -\Gamma_k \) to \( K \);
    \item If \( +\Gamma_j \in K \) and the edge is negative, assign \( +\Gamma_k \) to \( K \);
    \item If \( -\Gamma_j \in K \) and the edge is positive, assign \( +\Gamma_k \) to \( K \);
    \item If \( -\Gamma_j \in K \) and the edge is negative, assign \( -\Gamma_k \) to \( K \).
\end{enumerate}
If multiple such \( \Gamma_j \in K \) are adjacent to \( \Gamma_k \), we arbitrarily choose one. Repeat until every reaction vector is assigned a signed representative in \( K \). Since the R-graph is connected and finite, this process terminates.

Now suppose, for contradiction, that Corollary~\ref{lem:reac_vec_extreme ray} is not satisfied. That is, there exist reaction vectors \( \Gamma_i \) and \( \Gamma_j \) such that:
\begin{enumerate}
    \item \( \Gamma_j \in Q_1^+(\Gamma_i) \), and
    \item both \( \Gamma_i \) and \( \Gamma_j \) are included in \( K \) (with the given signs).
\end{enumerate}

Then the edge between \( \Gamma_i \) and \( \Gamma_j \) in the R-graph must be positive. Since both were assigned signs based on parity of the number of positive edges from \( \Gamma_1 \), they lie in the same parity class: both reached from \( \Gamma_1 \) via a path with an even number of positive edges.

Concatenating the path from \( \Gamma_1 \) to \( \Gamma_i \), the path from \( \Gamma_1 \) to \( \Gamma_j \) (traversed in reverse), and the positive edge between \( \Gamma_i \) and \( \Gamma_j \), we obtain a closed circuit with an odd number of positive edges—a contradiction to the signed loop property.

Similar contradictions arise in all other cases where a pair of reaction vectors with the wrong relative signs (e.g., \( -\Gamma_i \in K \) and \( \Gamma_j \in K \cap Q_1^-(\Gamma_i) \)) would violate Corollary~\ref{lem:reac_vec_extreme ray}.

Thus, all sign assignments are consistent with the corollary, and the constructed cone \( K \) ensures the system is monotone with respect to \( K \).
\end{proof}

Note the cone from above might not always be pointed. We show that if the reaction vectors are linearly independent, the cone must always be pointed.

\begin{cor}
Assume the reaction network in Lemma \ref{lem:reaction_coords_independent} has linearly independent reaction vectors. Then the cone produced by that lemma is, when restricted to the stoichiometric compatibility class, proper, convex, and pointed.
\end{cor}

\begin{proof}
The cone \( K \) from Lemma \ref{lem:reaction_coords_independent} is generated by linearly independent reaction vectors that span the stoichiometric subspace. 

Convexity follows directly from the fact that \( K \) is generated as a conical combination of vectors.

Properness holds because the dimension of \( K \) matches that of the stoichiometric compatibility class, which is equal to the number of linearly independent reaction vectors.

To show pointedness, suppose for contradiction that \( x \in K \) and \( -x \in K \) for some nonzero \( x \). Then there exist nonnegative coefficients \( a_i, b_i \geq 0 \), not all zero, such that
\[
x = \sum_i a_i \Gamma_i, \quad -x = \sum_i b_i \Gamma_i.
\]
Adding these gives:
\[
0 = \sum_i (a_i + b_i) \Gamma_i.
\]
Since not all \( a_i + b_i \) are zero and the \( \Gamma_i \) are linearly independent, this contradicts the assumption. Therefore, \( K \) must be pointed.
\end{proof}

Next we assume our reaction network has linearly dependent reaction vectors:

\begin{lem}
\label{lem:monotone_in_react_coords_2_entries_per_row}
Suppose a reaction network is monotone in reaction coordinates and the reaction vectors are linearly dependent. Then, by multiplying each reaction vector by a nonzero scalar, we can obtain a stoichiometric matrix in which every row has exactly two nonzero entries that sum to zero.
\end{lem}

\begin{proof}
We begin by applying the same procedure as in Lemma \ref{lem:reaction_coords_independent}, adjusting the signs of the reaction vectors so that each \( \Gamma_i \) is the only vector in \( Q_1^+(\Gamma_i) \). Under this construction, each row of the resulting stoichiometric matrix must have at most two nonzero entries, and these entries must have opposite signs.

Indeed, suppose a row had three nonzero entries. Then the corresponding three reactions would form a triangle in the \( R \)-graph. If all three entries had the same sign, then each edge in the triangle would be positive. If only two entries had the same sign, then one edge would be positive and two would be negative. In either case, the signed loop property would be violated, contradicting monotonicity.

Now, since the reaction vectors are linearly dependent, there exists a nontrivial linear combination:
\[
\sum_{i=1}^n a_i \Gamma_i = 0,
\]
with at least one \( a_i \neq 0 \). Define new reaction vectors \( \Gamma_i' = a_i \Gamma_i \). Then we have:
\[
\sum_{i=1}^n \Gamma_i' = 0.
\]

We now argue that none of the \( \Gamma_i' \) can be zero. Suppose, for contradiction, that \( \Gamma_j' = 0 \). Then any other reaction vector \( \Gamma_k \) sharing a nonzero coordinate with \( \Gamma_j \) must also have \( a_k = 0 \) to preserve cancellation. Repeating this logic across all connected reactions would force \( a_i = 0 \) for all \( i \), contradicting our assumption of linear dependence. Hence, all \( \Gamma_i' \) are nonzero.

Furthermore, since each row initially had at most two nonzero entries with opposite signs, and the \( \Gamma_i' \) are scalar multiples of the \( \Gamma_i \), each row in the new matrix has exactly two nonzero entries with equal magnitude and opposite sign. Thus, each row sums to zero.
\end{proof}

\begin{lem}
\label{lem:react coord alignable}
If a reaction network is linearly dependent and monotone in reaction coordinates, then its stoichiometric matrix \( \Gamma \) is alignable.
\end{lem}

\begin{proof}
Let the reaction vectors be \( \{R_i\}_{1 \leq i \leq n} \), with stoichiometric matrix \( \Gamma \). Since the system is monotone in reaction coordinates, the \( R \)-graph condition holds, implying that each species appears in at most two reactions.

Suppose, for contradiction, that some species appears in only one reaction \( R_i \). Because the set \( \{R_i\} \) is linearly dependent, there must exist a nontrivial linear combination \( \sum_{i=1}^n a_i R_i = 0 \). However, the unique nonzero entry for that species cannot be canceled unless \( a_i = 0 \), forcing \( R_i \) to be excluded. Moreover, any reaction sharing a species with \( R_i \) must also be excluded, as each species appears in at most two reactions. By connectedness of the reaction graph, this exclusion propagates to all reactions, implying \( a_i = 0 \) for all \( i \), a contradiction. Hence, every species must appear in exactly two reactions.

Now consider the linear dependence relation \( \sum_{i=1}^n a_i R_i = 0 \). By the argument above, each \( a_i \neq 0 \). Define a diagonal matrix \( D \) with entries \( D_{ii} = 1/a_i \), and let \( M \) be the matrix whose columns are \( a_i R_i \), so that \( \Gamma = M D \).

Each row of \( M \) has exactly two nonzero entries with opposite signs. Moreover, if two reactions \( R_i \) and \( R_j \) share a set of species \( I \), then the linear relation implies that \( (a_i R_i)_I = -(a_j R_j)_I \). Thus, the hypotheses of Lemma \ref{lem:cube type to aligned} are satisfied, and \( M \) admits a factorization \( M = P N \), where \( N \) is aligned. It follows that \( \Gamma = P N D \), and so \( \Gamma \) is alignable.
\end{proof}

\begin{thm}
\label{thm:dependent reaction monotone is contractive}
Suppose a reaction network is monotone in reaction coordinates and has linearly dependent reaction vectors. Then the network is weakly contractive.
\end{thm}

\begin{proof}
Assume without loss of generality that the stoichiometric matrix \( \Gamma \) is aligned and has entries in \( \{-1, 0, 1\} \). (If not, we may make it aligned, as in Lemma \ref{lem:react coord alignable}.) Also assume, following Lemma \ref{lem:monotone_in_react_coords_2_entries_per_row}, that each row of \( \Gamma \) has exactly one positive and one negative entry, or is identically zero.

Let \( \{R_i\}_{1 \leq i \leq n} \) denote the set of reaction vectors. Consider the set \( M \) consisting of all subset sums of these vectors, i.e., all vectors of the form \( \sum_{j \in A} R_j \) for \( A \subseteq \{1, \dots, n\} \). Let \( \operatorname{conv}(M) \) be the convex hull of this set, and let \( M' \subset \operatorname{conv}(M) \) denote the set of its extreme points.

We claim that \( M' \) is a viable set. It suffices to show that \( M \) is a viable set:

Since \( \Gamma \) is aligned, each \( R_i \in \{-1, 0, 1\}^n \). For any subset sum \( s = \sum_{j \in A} R_j \), we have \( s \in Q_1(R_i) \) if \( i \in A \), and \( s \in Q_1(-R_i) \) otherwise. Thus, all elements of \( M \) lie in the appropriate regions to be considered permissible.

To see this explicitly, suppose \( i \in A \), so that
\[
s = R_i + \sum_{j \in A \setminus \{i\}} R_j.
\]
Then for each species affected by \( R_i \), the corresponding component of \( s \) is either zero or has the same sign as in \( R_i \), since each \( R_j \in \{-1, 0, 1\}^n \), and each row of \( \Gamma \) has exactly two nonzero entries of equal magnitude and opposite sign. This implies that when $i \in A$ then $s \in Q_1(R_i)$, and when $i \not\in A$ then $s \in Q_1(-R_i)$.

Moreover, \( M \) is closed under permissible operations. Suppose \( s = \sum_{j \in A} R_j \in M \), and consider a permissible operation involving \( R_i \). If \( i \in A \), then $s \in Q_1(R_i)$ and
\[
s - R_i = \sum_{j \in A \setminus \{i\}} R_j \in M;
\]
if \( i \notin A \), then $s \in Q_1(-R_i)$ and
\[
s + R_i = \sum_{j \in A \cup \{i\}} R_j \in M.
\]
Thus, every permissible operation applied to an element of \( M \) yields another element of \( M \), so \( M \) is closed under permissible operations.

Therefore, \( M' \) is a viable set. By Theorem \ref{thm:output_M_to_strong_monotone}, the lifted system is strongly monotone. Applying Theorem \ref{thm:strongly_monotone_lift_to_weakly contractive_general_2}, we conclude that the system is weakly contractive on the relative interior of each stoichiometric compatibility class.

Finally, by Corollary \ref{cor:align to general} and Lemma \ref{lem:react coord alignable}, the result extends to the original network, even if \( \Gamma \) is not aligned. Hence, the network is weakly contractive.
\end{proof}

This result is related to the results in \cite{Angeli2010-hj} and \cite{ANGELI2008128}. This result also shows the global convergence of this class of networks in \cite{Angeli2010-hj} using the results in \cite{ANGELI2008128}. One novel aspect of our result, however, is the establishment of a decreasing norm on the species space of our reaction network (note that in \cite{ANGELI2008128} they do establish a function which decreases along the trajectories, though it is constructed differently and only applied to reaction coordinates in \cite{Angeli2010-hj}).

Finally, we can prove the following:

\begin{thm}
Any reaction network that is monotone in reaction coordinates with respect to an orthant cone is also monotone in species coordinates (possibly after adding a dummy species).
\end{thm}

\begin{proof}
By Lemma~\ref{lem:reaction_coords_independent}, if the reaction vectors are linearly independent, the system is monotone with respect to a cone in species coordinates. If the reaction vectors are linearly dependent, then by the proof of Theorem~\ref{thm:dependent reaction monotone is contractive}, the lifted system is strongly monotone, and thus the original system is monotone in species coordinates as well.
\end{proof}

Lastly, we will note that if we combine these results with the results on cubical cones, we arrive at:

\begin{thm}
\label{2 in each row aligned weakly contractive}
Suppose \( \Gamma \) is an aligned matrix in which each row has at most two nonzero entries. Then the corresponding system is weakly contractive with respect to some norm.
\end{thm}

\begin{proof}
If the columns of \( \Gamma \) are linearly independent, the result follows from Theorem \ref{thm:cubicals are weakly contractive}. If they are linearly dependent, then Theorem \ref{thm:dependent reaction monotone is contractive} applies.

Indeed, any aligned matrix with linearly dependent columns and at most two nonzero entries per row must have either zero or exactly two nonzero entries per row. As observed in the proof of Lemma \ref{lem:react coord alignable}, if a row had only one nonzero entry, the corresponding reaction vectors would necessarily be linearly independent. Thus, we may assume without loss of generality that each row has exactly two nonzero entries.

Following the proof of Lemma \ref{lem:monotone_in_react_coords_2_entries_per_row}, we may multiply each column by \( \pm 1 \) so that every row has exactly one \( +1 \) and one \( -1 \). In this case, the associated \( R \)-graph contains no positive edges and therefore automatically satisfies the \( R \)-graph condition. It then follows from Theorem \ref{thm:dependent reaction monotone is contractive} that the system is weakly contractive.
\end{proof}

Note that in the course of the previous proof, we in fact showed that an aligned matrix with at most two nonzero entries per row and linearly dependent columns must satisfy the \( R \)-graph condition. This leads to the following corollary:

\begin{cor}
\label{cor:2 entries per row iff r graph}
An aligned matrix \( N \) with linearly dependent columns satisfies the \( R \)-graph condition if and only if it has at most two nonzero entries per row.
\end{cor}

\begin{proof}
By Lemma~\ref{lem:monotone_in_react_coords_2_entries_per_row}, if \( N \) satisfies the \( R \)-graph condition, then it must have exactly two nonzero entries per row. Conversely, by the proof of Theorem~\ref{2 in each row aligned weakly contractive}, if \( N \) is aligned, has linearly dependent columns, and at most two nonzero entries per row, then it must in fact have exactly two nonzero entries per row, and satisfies the \( R \)-graph condition. The result follows.
\end{proof}

We will refer to a network $\Gamma$ which satisfies the $R$-graph condition and has linearly dependent columns as a \textbf{type A} network. 

\begin{cor}
    A reaction network is factorizable as $PND$ where $N$ has at most two entries per row and linearly dependent columns, if and only if it it is type A.
\end{cor}

\begin{proof}
    Since the $R$-graph does not depend on reversibility of the reactions, the $R$-graph of $PND$ and $N$ are the same. Since $P$ is injective $PND$ has linearly dependent columns if and only if $N$ does.

    First assume $N$ has at most two entries per row and linearly dependent columns. By Corollary \ref{cor:2 entries per row iff r graph} $N$ satisfies the $R$-graph condition and as we noted we also must have the $\Gamma = PND$ has linearly dependent columns, so it is type A. 

    Now assume the network is type $A$. By Lemma \ref{lem:react coord alignable} $A
    \Gamma$ is alignable, with some alignement $N$. Note as we noted earlier we must have $N$ also satisfy the $R$-graph condition, and have lienarly dependent columns. Again by Corollary \ref{cor:2 entries per row iff r graph} we have that $N$ has at most two nonzero entries per row, and so we are done.
\end{proof}

\section{Proof of the Theorem \ref{thm:collection of theorems}}
\label{sec:proof of theorem 1}

\begin{enumerate}
    \item Let \( S \) be the set of matrices whose entries lie in \( \{-1, 0, 1\} \) and which have at most two nonzero entries in each column (hence all matrices in \( S \) are alignable by Lemma~\ref{lem:cube type to aligned}).
    \item Let \( S^t \) be the set of matrices such that \( A \in S^t \) if and only if \( A^t \in S \).
    \item Let \( \mathcal{N} = S \cup S^t \).
    \item Let \( \mathcal{P} \) be the set of matrices with at most one nonzero entry in each row and no column of all zeros.
    \item Let \( \mathcal{D} \) be the set of square diagonal matrices with strictly positive diagonal entries and zeros elsewhere.
\end{enumerate}

We now prove Theorem~\ref{thm:collection of theorems} (restated below for convenience):

\begin{fixedtheorem}
Let a non-catalytic reaction network have a stoichiometric matrix of the form \( PND \), where \( P \in \mathcal{P} \), \( N \in \mathcal{N} \), and \( D \in \mathcal{D} \), and where the matrix product \( PND \) is well-defined. If, in addition, the \( \mathcal{RI} \)-graph of the network is strongly connected, then the corresponding system is weakly contractive.
\end{fixedtheorem}

\begin{proof}
Multiplying the reaction vectors by a nonzero scalar (via the diagonal matrix \( D \)) does not affect any of the contractivity conclusions. Under our kinetic assumptions, the systems defined by \( PN \) and \( PND \) generate the same family of differential equations. Thus, \( PN \) is weakly contractive if and only if \( PND \) is.

First, suppose \( N \in S^t \). Then, by Lemma~\ref{lem:cube type to aligned}, we can write \( PN = PP'N' \), where \( P' \in \mathcal{P} \) and \( N' \in S^t \subseteq \mathcal{N} \) is aligned. Since \( PP' \in \mathcal{P} \), the result follows by applying Theorem~\ref{2 in each row aligned weakly contractive} and Theorem~\ref{thm:P*gamma network strongly monotone} to the matrix \( PP'N' \).

Now suppose \( N \in S \). In this case, we can find a diagonal matrix \( D \in \mathcal{D} \) such that \( ND \) is of type C (for example, by multiplying each column of \( N \) that has a single nonzero entry by 2). Then, by Corollary~\ref{cor:type C is weakly contractive}, \( PND \) is weakly contractive. This completes the proof.
\end{proof}

\section{Algorithm to Check the Factorization}

A natural question is the following: given a matrix \( \Gamma \), can we factor it as \( \Gamma = PND \)? We now present a necessary and sufficient algorithm for the existence of such a factorization. (Code available at \url{https://github.com/alon-duvall/testing_PND_factorization}.)

Note that if \( N \in S \), then \( N \) is aligned, and if \( N \in S^t \), then by Lemma~\ref{lem:cube type to aligned}, it is alignable. Thus, we will assume that we are given a factorization \( \Gamma = PND \) as in Theorem~\ref{thm:collection of theorems}, where \( N \) is aligned.

\begin{lem}
\label{lem:equivalence of row relation matrices}
The nonzero indices in each column of the matrix \( P \) correspond bijectively with an equivalence class of the rows of \( \Gamma \), as defined in Definition~\ref{def:equivalence for rows}.
\end{lem}

\begin{proof}
First, observe that the diagonal matrix \( D \) does not affect the equivalence relation among the rows of \( \Gamma \), so we may assume without loss of generality that \( D \) is the identity matrix.

If any two rows of \( N \) belonged to the same equivalence class, this would contradict the assumption that \( N \) is aligned: two columns would then have two rows with either the same or opposite signs, violating the definition of an aligned matrix. Thus, in an aligned matrix \( N \), each equivalence class contains exactly one row.

Therefore, when computing the product \( PN \), each column of \( P \) corresponds to a distinct equivalence class of the rows of \( \Gamma \), and these equivalence classes are precisely those determined by Definition~\ref{def:equivalence for rows}. In other words, the equivalence classes of the rows of \( \Gamma \) coincide with the nonzero row indices of the columns of \( P \) when \( N \) is aligned.
\end{proof}

Hence, when a factorization \( \Gamma = PND \) exists with \( N \) aligned, we can proceed by first computing the equivalence classes of the rows of \( \Gamma \) using Definition~\ref{def:equivalence for rows}. These equivalence classes determine, for each column of \( P \), which rows are nonzero.

Moreover, for a given equivalence class, the corresponding rows in \( \Gamma = PN \) must be scalar multiples of one another; that is, they span a one-dimensional subspace. Therefore, after determining the equivalence classes, we can select one column of \( \Gamma \) whose support matches the desired row indices as an initial guess for the corresponding column of \( P \); denote this initial guess by \( P' \).

\begin{lem}
If \( \Gamma = PND \), where \( P \) has disjoint column supports and \( N \) is aligned, then \( \Gamma \) can also be factored as \( \Gamma = P'N' \), where \( P' \) has columns corresponding to the equivalence classes of the rows of \( \Gamma \), and each column is constructed by selecting a representative column from \( \Gamma \). The matrices \( P \) and \( P' \) are the same up to a permutation of columns and nonzero scalar multiples of their columns.
\end{lem}

\begin{proof}
If \( \Gamma = PND \) with \( N \) aligned, then each row of \( N \) lies in a distinct equivalence class, and the same holds for \( ND \), since scaling by a diagonal matrix does not alter the equivalence class structure. Thus, \( P \) encodes the equivalence classes of the rows of \( \Gamma \). 

Moreover, we can write
\[
PND = (P D')(D'^{-1} N D)
\]
for any diagonal matrix \( D' \) with nonzero entries. This gives an alternative factorization \( \Gamma = P'N' \), where \( P' \) consists of some nonzero columns of \( P \), potentially scaled. Since permuting the columns of \( P' \) and the corresponding rows of \( N' \) does not affect the product \( P'N' \), we conclude that \( P \) and \( P' \) differ only by a column permutation and column-wise scaling.
\end{proof}

\begin{lem}
Suppose we can factor \( \Gamma = P'N' \), where \( P' \) encodes the row equivalence relation. Then \( \Gamma \) admits a factorization \( \Gamma = PND \) if and only if there exist positive diagonal matrices \( D_1 \) and \( D_2 \) such that \( D_1 N' D_2 = N'' \), where \( N'' \in S \cup S^t \).
\end{lem}

\begin{proof}
If we can factor \( \Gamma = P' D_1^{-1} N'' D_2^{-1} \), then this is already a factorization of the form \( PND \), with \( P = P' D_1^{-1} \), \( N = N'' \), and \( D = D_2^{-1} \).

Conversely, suppose we are given a factorization \( \Gamma = PND \). By Lemma~\ref{lem:equivalence of row relation matrices}, the matrix \( P \) encodes the same row equivalence relation as \( P' \), so \( P \) and \( P' \) must be the same up to column scaling and permutation. Permute the columns of \( P' \) and the corresponding rows of \( N' \) so that \( P \) and \( P' \) differ only by a positive diagonal scaling and $P'N'$ remains unchanged. Then there exists a positive diagonal matrix \( D_1 \) such that \( P = P' D_1^{-1} \). Substituting into the factorization, we have:
\[
P'N' = P D_1 N' = PND.
\]
Rewriting, we obtain:
\[
P D_1 N' D^{-1} = PN.
\]
Since \( P \) is injective (i.e., has full column rank), it follows that
\[
D_1 N' D^{-1} = N.
\]
Because \( N \in S \cup S^t \), we conclude that \( D_1 N' D^{-1} \in S \cup S^t \), as desired.
\end{proof}

\begin{algorithm}
\caption{Algorithm to check whether a matrix \( \Gamma \) can be factored as \( PND \)}\label{alg:algopnd}
\begin{algorithmic} 
\Require Matrix \( \Gamma \)

\State Compute the equivalence classes of the rows of \( \Gamma \) as in Definition~\ref{def:equivalence for rows}.
\State Construct a factorization \( \Gamma = P'N' \), where \( P' \) has one column per equivalence class, with nonzero entries only in the corresponding rows. Set the values in each column of \( P' \) by arbitrarily selecting a representative column from \( \Gamma \).
\If {the submatrix of \( \Gamma \) restricted to the rows of any equivalence class has column space of dimension two or more}
    \State \Return No factorization exists.
\EndIf

\State Attempt to find diagonal matrices \( D' \) and \( D'' \) such that \( D'N D'' \) consists only of entries in \( \{ -1, 0, 1 \} \).
\State Normalize by assuming the first entry of \( D' \) is 1, and recursively solve for the remaining entries using known values in shared rows or columns of \( N \).
\If {this recursive computation produces inconsistent values}
    \State \Return No factorization exists.
\EndIf

\State Check whether the resulting matrix \( N \in S \cup S^t \).
\If {this condition fails}
    \State \Return No factorization exists.
\EndIf

\State Let \( N \) be the matrix \( D'ND'' \), now with entries in \( \{ -1, 0, 1 \} \).
\State Set \( P = P'(D')^{-1} \) and \( D = (D'')^{-1} \).
\State \Return Factorization successful; return matrices \( P \), \( N \), and \( D \).
\end{algorithmic}
\end{algorithm}

Based on this, our algorithm proceeds by first factoring \( \Gamma = P'N' \) through the equivalence relation on the rows of \( \Gamma \). We then need to determine whether there exist diagonal matrices \( D' \) and \( D \) such that
\[
P'N' = P D' N D
\]
with \( N \) aligned. This is equivalent to checking whether there exist diagonal matrices \( D' \) and \( D \), both with strictly positive entries, such that \( D'ND \) has all entries in \( \{ -1, 0, 1 \} \). In this case, \( D'ND \) would be an aligned matrix.

Without loss of generality, we can assume all entries of \( D' \) and \( D \) are positive. Moreover, we may rescale them by any nonzero scalar \( r \) and \( 1/r \), respectively, so we can assume that the first diagonal entry of \( D' \) is 1.

To determine whether the desired scaling is possible, we attempt to scale the entries of \( N \) inductively. We begin by choosing diagonal values that normalize the first row of \( N \) so that all of its nonzero entries have absolute value 1. We then propagate through the matrix, using the values already fixed to determine the required diagonal entries in \( D' \) and \( D \) that ensure each nonzero entry of \( D'ND \) is \( \pm 1 \).

If this procedure succeeds, we conclude that \( \Gamma \) can indeed be factored as \( \Gamma = P D' N D \), with \( N \) aligned. If the process fails at any step, then no such factorization \( PND \) exists.

Finally, we check whether \( N \in S \cup S^t \). This completes a necessary and sufficient computational procedure to determine whether \( \Gamma \) admits a factorization of the form \( \Gamma = PND \).

\section{Biochemical Examples}
\label{sec:biochemical examples}

Our theorems apply to a wide range of biochemical systems. Although these examples arise in different contexts and employ diverse mathematical techniques to establish convergence, we show that weak contractivity—and consequently, global convergence within stoichiometric compatibility classes—follows directly from Theorem \ref{thm:collection of theorems}. 

To conclude that all trajectories within a stoichiometric compatibility class converge to a single equilibrium, we must sometimes verify a technical condition known as \textit{persistence}, which ensures that no species is lost asymptotically. Even in the absence of persistence, weak contractivity significantly constrains the system’s long-term behavior.

\subsection{Processive Phosphorylation}

Our theorem applies to models of processive phosphorylation, as studied in \cite{EITHUN20171}. Phosphorylation is a biochemical process in which an enzyme facilitates the addition of phosphate groups to a substrate. A model from \cite{EITHUN20171} is given by:

\begin{align}
    &S_0 + K \rightleftharpoons S_0 K \rightarrow S_1 K \rightarrow \dots \rightarrow S_{n-1} K \rightarrow S_n + K, \\
    &S_n + F \rightleftharpoons S_n F \rightarrow \dots \rightarrow S_2 F \rightarrow S_1 F \rightarrow S_0 F.
\end{align}

The $\mathcal{RI}$-graph of this network is strongly connected, and each species appears at most twice, ensuring weak contractivity on the interior of \( \mathbb{R}^n_{\geq 0} \) by Theorem \ref{thm:collection of theorems}. This, along with persistence as established in \cite{EITHUN20171}, allows us to immediately conclude global convergence on the stoichiometric compatibility classes.

\subsection{PCR Kinetics: Primer Annealing}

The reaction network modeling the primer annealing step in PCR kinetics, as given in \cite{https://doi.org/10.1002/bit.20617}, also satisfies our conditions. In PCR, double-stranded DNA is denatured, producing two template strands \( T_1 \) and \( T_2 \) to which primers \( P_1 \) and \( P_2 \) anneal. The annealing reactions are:

\begin{align}
    P_1 + T_1 &\rightleftharpoons H_1, \\
    P_2 + T_2 &\rightleftharpoons H_2, \\
    T_1 + T_2 &\rightleftharpoons U, \\
    P_1 + P_2 &\rightleftharpoons D.
\end{align}

Since each species appears at most twice, Theorem \ref{thm:collection of theorems} applies, and the network is weakly contractive on the interior of \( \mathbb{R}^n_{\geq 0} \). By Corollary \ref{cor:global fixed point} this system is globally convergent.

\subsection{RKIP Network}

A model of the Raf-1 kinase inhibitor protein (RKIP), part of the EGF signaling pathway, is studied in \cite{Angeli2010-hj}. The corresponding reaction network is:

\begin{align}
    \mbox{Raf-1} + \mbox{RKIP} &\leftrightarrow \mbox{Raf-1/RKIP},\\
    \mbox{Raf-1/RKIP} + \mbox{ERK-PP} &\leftrightarrow \mbox{Raf-1/RKIP/ERK-PP}, \\
    \mbox{Raf-1/RKIP/ERK-PP} &\rightarrow \mbox{Raf-1} + \mbox{ERK} + \mbox{RKIP-P}, \\
    \mbox{MEK-PP} + \mbox{ERK} &\leftrightarrow \mbox{MEK-PP/ERK}, \\
    \mbox{MEK-PP/ERK} &\rightarrow \mbox{MEK-PP} + \mbox{ERK-PP}, \\
    \mbox{RKIP-P} + \mbox{RP} &\leftrightarrow \mbox{RKIP-P/RP}, \\
    \mbox{RKIP-P/RP} &\rightarrow \mbox{RKIP} + \mbox{RP}.
\end{align}

In \cite{Angeli2010-hj}, this network is shown to be persistent. Since each species appears at most twice and \( \Gamma \) has entries in \( \{-1,0,1\} \), the network is alignable. By Theorem \ref{thm:collection of theorems}, it is weakly contractive on the interior of \( \mathbb{R}^n_{\geq 0} \), recovering the result from \cite{Angeli2010-hj} that all initial conditions in \( \operatorname{int}(\mathbb{R}_{\geq 0}^{11}) \) converge to a unique equilibrium.

\subsection{Electron Transfer Networks}

Electron transfer networks, as modeled in \cite{article}, describe systems where substrates exist in reduced or oxidized states and exchange electrons upon interaction. These models arise in biological contexts, such as mitochondrial electron transport. The general reaction form is:

\[
B_i + A_j \rightleftharpoons A_i + B_j.
\]

For any number of species \( A_i, B_i \), these networks can be factorized as \( PND \) by treating \( \{B_i, -A_i\} \) as a single species for each \( i \), where \( N \) corresponds to a type C network. It follows that these networks are weakly contractive. By Corollary \ref{cor:global fixed point} this system is globally convergent.

\section{Discussion}

We have introduced new classes of monotone systems, along with a novel method for proving global convergence via contractivity. While monotonicity and contractivity are typically viewed as distinct tools for analyzing the qualitative behavior of dynamical systems, this paper highlights a useful connection between them. Our approach reveals that many existing results in the literature share a common structural feature—namely, the factorization described in Theorem~\ref{thm:collection of theorems}. An interesting direction for future work is to determine precisely how many aligned matrices give rise to weakly contractive systems.

\appendix

\section{Polytope Geometry}
\label{section_appendix_polytope_terms}
Here we lay out a few basic definitions related to polytopes and cones.

\begin{define}
    A \textit{polytope} is the convex hull of a finite number of points in $\mathbb{R}^n$. A \textit{cone} is a set of points in $\mathbb{R}^n$ which is closed under addition and multiplication by nonnegative scalars. 
\end{define}

\begin{define}
    The \textit{affine hull} of a set $S \subset \mathbb{R}^n$ is the set $\mbox{affHull}(S) = \{ \sum_i \alpha_i x_i \st  x_i \in S, \alpha_i \in \mathbb{R}, \sum_i \alpha_i = 1 \}$. The \textit{dimension} of a convex set is the dimension of its affine hull. The \textit{relative interior} of a set $S$ is its interior in the subspace topology of $\mbox{affHull}(S)$.
\end{define}

A \textit{ray} is a one-dimensional cone.

\begin{define}
    An \textit{extreme point} of a polytope $P$ is any point $x \in P$ that satisfies if there exists $y,z \in P$ and $0 \leq \gamma \leq 1$ such that $x = \gamma y + (1-\gamma)z $, then $x=y=z$.
\end{define}

\begin{define}
    Suppose we have a cone $K$. An \textit{extreme ray} $k$ of $K$ is a ray contained in  $K$, satisfying the property that if $\exists x,y \in K$ and $0 \leq \gamma \leq 1$ such that $\gamma x + (1-\gamma) y \in k$, then we must have $x,y \in k$.
\end{define}

\begin{define}
    A \textit{face} of a polytope or cone \( P \) is any set of the form \( P \cap \mathcal{H} \), where \( \mathcal{H} \) is a supporting hyperplane of \( P \), i.e., a hyperplane that does not intersect the relative interior of \( P \).

    A \textit{proper face} of a cone \( K \) is a face of the form \( K \cap \mathcal{H} \) that is not equal to \( K \) itself and does not consist solely of the origin.
\end{define}

\begin{define}
    We define a \textit{cube} to be the set of $2^n$ vectors in $\mathbb{R}^n$ consisting of every possible vector $v_i$ such that for all $1 \leq j \leq n$ we have that $(v_i)_j \in \{0,1\}$. Here $n$ can be any positive integer.
\end{define}

\begin{define}
    We say a polytope $K \subseteq \mathbb{R}^k$ is \textit{cubical} if there exists an injective affine transformation $T: \mathbb{R}^k \rightarrow \mathbb{R}^n$ such that $T(K)$ is a cube.
\end{define}

\begin{define}
    A \textit{cubical cone} is any pointed and polyhedral cone such that it has a cubical cross-section.
\end{define}

We will need a lemma to characterize a certain operation:

\begin{lem}
\label{lem:extrem points characterize}

The following two operations on a finite set of points $ A =\{v_1,v_2,..., v_n\} \subset \mathbb{R}^n$ are equivalent:

\begin{enumerate}
    \item Repeatedly remove all points $v_i$ that can be expressed as a convex combination of the other points (removing the points in an unspecified order).
    \item Take the convex hull of $A$ and then take the extreme points of this convex hull
\end{enumerate}
    
\end{lem}

\begin{proof}
    To show that the set from 2 is contained in the set from 1, note that every point that is a convex combination of the other points cannot be an extreme point.

    Next, we will show the set of points from 1 is contained in the set of points from 2 (no matter the order in which we remove points). Refer to the set of the remaining set of points as $B$. Note all the points in $A$ were a convex combination of the points in $B$. This can be seen inductively: The last point to be removed was a convex combination of $B$, the second to last a convex combination of $B$ and the last point, and so forth. Since the last point was a convex combination of $B$, the second to last is also a convex combination of $B$. Thus, by induction, every point is in the convex combination of $B$.
    
    Every point remaining in $B$, after removing as many points as possible, cannot be a convex combination of the other vectors in $A$. If it were, then it would be a convex combination of points in $B$ and we could remove it. Thus each point in $B$ is not in the convex combination of the other points and is thus an extreme point in the convex hull of $B$. Since the convex hull of $B$ is the same as that of $A$, each of these points is an extreme point. Thus the set from operation 1 is always contained in the set from operation 2, and so they are the same set.

    \end{proof}

\section*{Acknowledgment}
The authors would like to thank Polly Yu, Murad Banaji, Yuzhen Fan, Aria Masoomi, David Massey, and Zahra Aminzare for helpful comments and conversations about the paper.
This work was partially supported by grants AFOSR FA9550-21-1-0289 and NSF/DMS-2052455.

Generative AI was used to improve the readability of the paper.

\bibliographystyle{unsrt}
\bibliography{references}

\end{document}